\documentclass[10pt,a4paper]{amsart}

\usepackage[utf8]{inputenc}
\usepackage[T1]{fontenc}
\usepackage{lmodern}

\usepackage{fullpage} 

\usepackage{amsmath}
\usepackage{physics}
\usepackage{color}

\usepackage[hidelinks]{hyperref}
\usepackage{hyperref}

\theoremstyle{plain}
\newtheorem{Lem}{Lemma}

\newcommand{\cm}[2]{\bqty{#1,#2}} 
\newcommand{\qn}[1]{\bqty{#1}_q} 
\newcommand{\proj}[2]{P^{#1}_{#2}}  
\newcommand{\qbc}[2]{ \left[\,\begin{matrix} #1 \\ #2 \end{matrix}\,\right]_q } 
\newcommand{\qbcs}[2]{ \left[\begin{smallmatrix} #1 \\ #2 \end{smallmatrix}\right]_q } 

\title[The $q$-analogue of the QTAM: a review from SF]{The $q$-analogue of the Quantum Theory of Angular Momentum: a review from Special Functions}
\author[R. \'Alvarez-Nodarse]{Renato \'Alvarez-Nodarse}
\address{Departamento de An\'alisis Matem\'atico, Universidad de Sevilla, Av.~Reina Mercedes s/n, E-41012 Sevilla, Spain}
\email{ran@us.es}
\author[A. Arenas-G\'omez]{Alberto Arenas-G\'omez$^{\ast}$}
\address{Departamento de Matem\'aticas y Computaci\'on, Universidad de La Rioja, Calle Madre de Dios 53, 26006 Logro\~no, Spain}
\email{alberto.arenas@unirioja.es}
\date{\today}
\thanks{$^{\ast}$ Corresponding author}

\begin{document}

\begin{abstract}
In the present paper we review the $q$-analogue of the Quantum Theory of Angular Momentum based on the $q$-algebra $su_q(2)$ with a special emphasis on the representation of the Clebsch-Gordan coefficients in terms of $q$-hypergeometric series. This representation allows us to obtain several known properties of the Clebsch-Gordan coefficients in an unified and simple way.
\end{abstract}


\maketitle


\section{Introduction}\label{sec1}
It is well known the important role that the representation theory of groups and algebras plays in Physics (see e.g.~\cite{ste:94,tik:03}), and in particular in the Quantum Theory of Angular Momentum (QTAM)~\cite{var:88}. In fact a deep knowledge of the group theory and, in particular, of the representation theory, allows us to understand a lot of phenomena of physical systems as it is shown, for example, in the already classical monographs~\cite{bie:81,bie:81b}.

So it is understandable that after the appearance of the so-called quantum groups and $q$-algebras at the end of the XX century, there were an increasing interest in understanding if they can have any role in solving some physical problems. In fact, quantum groups and their corresponding $q$-algebras appeared as an abstraction in the development of the theory of quantum integrable dynamical systems,  as well as in the context of exactly soluble models of statistical physics (for the history of that subject see the nice surveys~\cite{fad89,fad06}).

After the publication of the first $q$-analogues of physical systems (like the harmonic oscillator \cite{bie,mac}), there were a lot of research papers exploring the connection between those new models with the $q$-algebras. One of the most interesting study was the one devoted to the construction of a suitable $q$-analogue of the Quantum Theory of Angular Momentum (for a review on the QTAM we refer to the reader to the classical books \cite{bie:81,var:88} and references therein).

Among all possible constructions of QTAM there is one based on the projection operators (PO), those defined by von Neumann in his seminal book \cite[p.~50]{von:18}. This approach has been used by Lowdin for the first time in 1964, further developed by Shapiro in 1965 for $su(2)$, and by Smirnov et al.~starting from 1968 for $su(3)$, etc.~(for a full history of this method see \cite{tol:05}). In fact, the construction of the $q$-analogue of the QTAM by using the projection operators method was developed by Smirnov, Tolstoy, and Kharitonov in a series of papers started with \cite{smi1,smi2}. Since these papers are hard to find, we will describe here the method used in \cite{smi1,smi2} for the $su_q(2)$ algebra and, in a future contribution, the case of the $su_q(1,1)$ as well. For a more recent review on the PO method see \cite{tol:01} and for the representation theory of $su_q(2)$ and  $su_q(1,1)$ $q$-algebras see~\cite{smi-g}.

We will also go further by exploring the connection of the Clebsch-Gordan coefficients (CGC) with certain special functions. The connection between $q$-algebras and $q$-special functions is well documented in the literature (see e.g.~\cite{koelink,koe:96,vil:92}). Here we will focus our attention in the connection with the $q$-hypergeometric function ${}_3F_2$ introduced by Nikiforov and Uvarov (see e.g.~\cite[p.~138]{nsu}), which is a
symmetric version of the well known basic hypergeometric series ${}_3\varphi_2$ (see e.g.~\cite{gas}). In fact, using some properties of the basic series, which can be easily translated to the $q$-hypergeometric function ${}_3F_2$, it is possible to obtain in a very easy way the properties of the corresponding CGC. As examples of this we will derive the symmetry properties of the CGC, properties that require an elaborate proof by other methods (see \cite{smi2}). Finally, we will establish the connection of the Clebsch-Gordan coefficients with a certain $q$-analogue of the Hahn polynomials by means of the obtained relation with the $q$-hypergeometric function~${}_3F_2$.

Our main aim in this paper is to offer an unified approach to the $q$-analogue of the Quantum Theory of Angular Momentum by means of the method of projection operators and the connection with the $q$-hypergeometric function ${_3}F_2(\cdot|q;z)$, which allow us to recover several well known results and obtain multiple new $q$-analogues of classical formulas presented in the handbook \cite{var:88}. The structure of the paper is the following: in Section \ref{sec_premResults} we will include the notation as well as the needed preliminary results that are relevant for our purposes; in Section \ref{sec_suq2Algebra} we will introduce the $q$-algebra $su_q(2)$ and we discuss some of its properties, including the construction of the projection operators that will allow us to obtain an explicit formula for computing the Clebsch-Gordan coefficients of this algebra; finally, in Section \ref{sec_CGC} we will obtain the explicit expressions for the CGC of the $su_q(2)$ $q$-algebra as well as several of their properties, including the symmetry properties and some recurrence relations.


\section{Some preliminary results}\label{sec_premResults}
Let be $q\in \mathbb{R}\setminus\{\pm1\}$ and $x\in\mathbb{R}$. The symmetric quantum number $[x]$ is given by
\begin{equation}
\qn{x}=\frac{q^{x}-q^{-x}}{q-q^{-1}}. \label{sqn}
\end{equation}
It is clear form the above equation that $\qn{x}\to x$ when $q\to1$.

Following \cite{nsu} we will define for all $a\in\mathbb{R}$, the symmetric $q$-Pochhammer by 
\begin{equation}
(a\vert q)_0:=1,\quad (a\vert q)_n = \prod_{m=0}^{n-1}\qn{a+m},\quad n=1,2,3,\dots
\label{sim-poc-q}
\end{equation}
The special case when $a=1$ leads to the symmetric $q$-factorial 
\begin{equation}
(1\vert q)_n=\qn{n}!:= \qn{1}\qn{2}\cdots \qn{n-1}\qn{n}.
\label{q-fac}
\end{equation}	
The symmetric $q$-factorial symbol $\qn{n}!$ can be expressed through the symmetric $q$-Gamma function 
$\tilde{\Gamma}(s)$ defined in \cite[Eq. (3.2.24) p.~67]{nsu} by the formula
$\qn{n}!=\tilde{\Gamma}(n+1)$, where 
\begin{equation}
\tilde{\Gamma}(s)=q^{-\frac{(s-1)(s-2)}{4}}\Gamma_q(s) ,
\label{tilG-q-Gamma}
\end{equation}
and the classical $q$-Gamma function $\Gamma_q(s)$ is given by \cite[Eq. (3.6.3) p.~79]{nsu}
\begin{equation}
\Gamma_q(s)=
\begin{cases}
(1-q)^{1-s} \dfrac{\prod_{k\geq0}(1-q^{k+1}) }{\prod_{k\geq0} (1-q^{s+k})}, & \text{if $0<q<1$}, \\
q^{\frac{(s-1)(s-2)}{2}} \Gamma_{1/q}(s), & \text{if $q>1$}.
\end{cases}
\label{q-gamma-clas}
\end{equation}

Notice that
$$
\lim_{q\to1}  \qn{n}!=n!,\quad    \lim_{q\to1} (a\vert q)_k = (a)_k,
$$ 
where $(a)_n$ is the Pochhammer symbol defined by
$$
(a)_0:=1,\quad (a)_k=a(a+1)\cdots(a+n-1),\qquad n\in\mathbb{N}.
$$
Moreover, if $a\in\mathbb{N}$ then
\begin{equation} \label{q-fac-poc}
(a\vert q)_n=\frac{\qn{a+n-1}!}{\qn{a-1}!}, \qquad
(-a\vert q)_n=
\begin{cases}
\dfrac{(-1)^n \, \qn{a}!}{\qn{a-n}!}, & \text{ if $a\geq n$}, \\
0, & \text{ if $a<n$}.
\end{cases}
\end{equation}
We will also use the $q$-symmetric analogue of the binomial coefficients
\begin{equation}\label{q-bin-coe}
\qbc{n}{k}:= \frac{\qn{n}!}{\qn{k}!\qn{n-k}!}.
\end{equation}

We will use the symmetric $q$-hypergeometric function introduced in \cite[p.~138]{nsu}
\begin{equation}   {}_{p+1}F_p
		\left(\!\!\begin{array}{c} {a_1,\dots,a_{p+1}} \\ {b_1,\dots,b_p} \end{array} \,
		\bigg\vert \, q\,,\,  z \right)\!=\!
		\displaystyle \sum _{k=0}^{\infty}\frac{ (a_1\vert q)_k(a_2\vert q)_k \cdots  (a_{p+1}\vert q)_k}
		{(b_1\vert q)_k(b_2\vert q)_k \cdots  (b_p\vert q)_k}\frac{z^k}{(1\vert q)_k},
		\label{q-hip-def}
\end{equation}
which is one of the $q$-analogues of the  the generalized hypergeometric function (see e.g.~\cite{aar})
\begin{equation}
	\begin{array}{l}
		\displaystyle\hspace{0.2cm}_{p+1} F _p\displaystyle
		\bigg(\begin{array}{c}{a_1,a_2,\dots,a_{p+1}} \\ {b_1,b_2,\dots,b_q}\end{array}
		\bigg\vert\, z\bigg)=\displaystyle\sum_{k=0}^{\infty}\frac{ (a_1)_k(a_2)_k
			\cdots (a_p)_k} {(b_1)_k(b_2)_k \cdots  (b_p)_k}\frac{z^k}{k!}.
	\end{array}
	\label{hip-def}
\end{equation}
For convenience in writing, we will also use the notation
$$
{}_{p+1} F _p \bigg({a_1,a_2,\dots,a_{p+1}}; {b_1,b_2,\dots,b_q}	\bigg\vert\, q,\, z\bigg):=	
\displaystyle\hspace{0.2cm}_{p+1} F _p\displaystyle
\bigg(\begin{array}{c}{a_1,a_2,\dots,a_{p+1}} \\ {b_1,b_2,\dots,b_q}\end{array}
\bigg\vert\, q,\, z\bigg),
$$
and a similar one for the case $q=1$.

In our work we will restrict ourselves to the case when one of the $a_i$, $i=1,2,\dots,p+1$, is a negative integer, so the series \eqref{q-hip-def} is always terminating. Also notice that when one of the $a_i=0$, the 
series \eqref{q-hip-def} is equal to~1.

Before continue, it is convenient to point out that there is another $q$-analogue of the hypergeometric function \eqref{hip-def}, the so-called basic (hypergeometric) series defined by
\begin{equation}
	{}_{p+1}\varphi_p \left( \begin{array}{c} {a_1,a_2,\dots,a_{p+1}} \\
		{b_1,b_2,\dots,b_p} 
	\end{array} \,\bigg\vert\, q\,,\, z \right)=
	\displaystyle \sum _{k=0}^{\infty}\frac{ (a_1;q)_k \cdots  (a_{p+1};q)_k}{(b_1;q)_k  \cdots  (b_p;q)_k}\frac{z^k}{(q;q)_k},
	\label{q-series-basicas}
\end{equation}
where 
\begin{equation}
	(a;q)_k= \prod_{m=0}^{k-1}(1-aq^m).
	\label{coef-serie-bas}
\end{equation}
There are two mainly reasons for using the $q$-hypergeometric series \eqref{q-hip-def} instead the basic series \eqref{q-series-basicas}. The first reason is that the former is invariant with respect to the change $q\to1/q$ (that is why it is called symmetric), and the second one is that we have the following straightforward limit
$$
\lim_{q\to1}{}
{}_{p+1}F_p
\left(\!\!\begin{array}{c} {a_1,\dots,a_{p+1}} \\ {b_1,\dots,b_p} \end{array} \,
\bigg\vert\, q\,,\,  z \right)\!=
{}_{p+1}F_p
\left(\!\!\begin{array}{c} {a_1,\dots,a_{p+1}} \\ {b_1,\dots,b_p} \end{array} \,
\bigg\vert\, z \right).
$$
Since the $q$-hypergeometric series  \eqref{q-hip-def} is related with the basic series \eqref{q-series-basicas} by the expression \cite[p.~139]{nsu}
\begin{equation}\label{connection}
	{}_{p+1}\varphi_p \left( \begin{array}{c} {q^{a_1},q^{a_2},\dots,q^{a_{p+1}}} \\
		{q^{b_1},q^{b_2},\dots,q^{b_p}} \end{array} \,\bigg\vert\, q\,,\, z \right)=
	{}_{p+1}F_p
	\left(\!\!\begin{array}{c} {a_1,\dots,a_{p+1}} \\ {b_1,\dots,b_p} \end{array} \,
	\bigg\vert\, q^{1/2}\,,\,  z q^{(a_1+\dots+a_{p+1}-b_1-\dots-b_p-1)/2}   \right),
\end{equation}
one can obtain several transformation and summation formulas for the  $q$-hypergeometric series from the already-known ones for the basic series. Notice that the symmetric $q$-hypergeometric function ${}_{p+1}F_p$ in the right hand side of the previous identity is defined for $q^{1/2}$, so this detail should be taking into account when ${}_{p+1}\varphi_p$ is evaluated on $z$ depending on $q$, such as in the case of certain summation and transformation formulas that will be used in the present work. 

We will start writing the following transformation formulas for the $q$-hypergeometric function ${}_{3}F_2$, which follows from Eqs.~(III.11) and (III.12) of \cite[p.~330]{gas}, respectively
\begin{equation}\label{142q}
{}_{3}F_2\left(\!\!\begin{array}{c} {-n,a,b} \\ {d, e} \end{array} \bigg\vert\, q,\, q^{\pm(a+b-n-d-e+1)} \right)\!=
\frac{q^{\pm an}(e-a\vert q)_n}{(e\vert q)_n}  
{}_{3}F_2\left(\!\!\begin{array}{c} {-n,a,d-b} \\ {d,a-e-n+1} \end{array}\bigg\vert\, q,\, q^{\pm(b-e)} \right),
\end{equation}
\begin{equation}\label{141q}
{}_{3}F_2\left(\!\!\begin{array}{c} {-n,a,b} \\ {d, e} \end{array} \bigg\vert\, q,\, q^{\pm(a+b-n-d-e+1)} \right)\!=
\frac{(d-a\vert q)_n(e-a\vert q)_n}{(d\vert q)_n(e\vert q)_n}  
{}_{3}F_2\left(\!\!\begin{array}{c} {-n,a,a+b-d-e-n+1} \\ {a-d-n+1,a-e-n+1} \end{array}\bigg\vert q,\, q^{\pm b} \right).
\end{equation}

We will also need some summation formulas. Using the $q$-analogue of the Vandermonde 
formula \cite[Eq. (1.5.3), p.~14]{gas} we have, for $n\in\mathbb{N}$,
\begin{equation}\label{q-van}
{}_{2}F_1\left(\!\!\begin{array}{c} {-n,b} \\ {c} \end{array} \bigg\vert\, q,\, q^{\pm(b-c-n+1)} \right)\!=
\frac{(c-b\vert q)_n}{(c\vert q)_n}q^{\pm nb},
\end{equation}
which has some restrictions that should be taken into account, specially in our work when both $b$ and $c$ are integer numbers. Fist of all, in the $_2F_1$ function \eqref{q-van} it is assumed that, if  $b$ is a negative integer, then $n<|b|$. Moreover, if $c$ is also a negative integer, then $n<\min(|b|,|c|)$, otherwise there will be a zero factor in the denominator of some terms of the series.

In the following we will consider $n,b,c\in\mathbb{Z}^+$, such that $n<\min(b,c)$. From \eqref{q-van} and  \eqref{q-fac-poc},
the following useful summation formula follows
\begin{equation}\label{summa1}
{}_{2}F_1\left(\!\!\begin{array}{c} {-n,b} \\ {c} \end{array} \bigg\vert\, q,\, q^{\pm(b-c-n+1)} \right)  =
\frac{\qn{c-b-1+n}!\qn{c-1}!}{\qn{c-b-1}!\qn{c-1+n}}q^{\pm bn},\quad c>b,
\end{equation}
or, equivalently, 
\begin{equation}\label{summa1equiv}
\sum_{r=0}^{n}\frac{(-1)^r \qn{b-1+r}!}{\qn{r}!\qn{c-1+r}!\qn{n-r}!}q^{\pm(b-c-n+1)r}
 =\frac{\qn{c-b-1+n}!\qn{b-1}!}{\qn{n}!\qn{c-b-1}!\qn{c-1+n}}q^{\pm bn},\quad c>b.
\end{equation}
If we make the change $c\to-c$ and $b\to -b$ in \eqref{q-van} and use \eqref{q-fac-poc}, we find the very useful 
formula 
\begin{equation}\label{suma2}
{}_{2}F_1\left(\!\!\begin{array}{c} {-n,-b} \\ {-c} \end{array} \bigg\vert\, q,  q^{\pm(b-c+n-1)} \right)  =
(-1)^n\dfrac{\qn{c-n}!\qn{b+n-c-1}!}{\qn{c}!\qn{b-c-1}!}q^{\pm bn},\quad b>c,
\end{equation}
or, equivalently, 
\begin{equation}\label{suma2equiv}
\sum_{r=0}^{n}\frac{(-1)^{r}\qn{c-r}!}{\qn{r}!\qn{b-r}!\qn{n-r}!}q^{\pm(b-c+n-1)r} =
 (-1)^n\dfrac{\qn{c-n}!\qn{b+n-c-1}!}{\qn{b}!\qn{n}!\qn{b-c-1}!}q^{\pm bn},\quad b>c,
\end{equation}
where we recall that $b$ and $c$ are positive integers.

Another important summation formula is the Jackson's terminating $q$-analogue of Dixon's sum \cite[Eq.~(II.15), p.~355]{gas}, that in terms of the symmetric $q$-hypergeometric function reads
\begin{equation}\label{q-dix}
{}_{3}F_2\left(\!\!\begin{array}{c} -2n\,\,,\,\,b\,\,,c\,\, \\ 1-2n-b,1-2n-c \end{array} \bigg\vert\, q,  q \right)\!=
\frac{q^n\qn{2n}!(b+c+n\vert q)_n}{\qn{n}!(b+n\vert q)_n(c+n\vert q)_n},\qquad n=0,1,2,\dots
\end{equation}

The following lemma, which can be proved by induction, will be useful in the next section. Recall that the commutator acts in two operators by $\cm{A}{B}=AB-BA$, and the quantum number of a operator by
$$
\qn{A}=\frac{q^{A}-q^{-A}}{q-q^{-1}}=\sum_{n=0}^{\infty}\frac{(\gamma A)^{n}}{(2n+1)!},
\qquad q=e^{\gamma}.
$$
\begin{Lem}\label{lem_aux1}
Let us consider three operators $A_+$, $A_-$, and $B$ such that $[B,A_\pm]=\pm A_\pm$, then
\begin{align}
\label{lem1}		B^r A_\pm &=A_\pm(B\pm I)^r, \\[10pt]
\label{lem2}		\qn{\nu B+\eta}A_\pm^r&=A_\pm^r\qn{\nu B+\eta\pm\nu r}, \qquad \nu,\eta\in\mathbb{R},
\end{align}
for all $r\in\mathbb{N}$. Moreover,
\begin{equation}
\label{lem3} 	\cm{B}{A_\pm^r}=\pm r A_\pm^r
\end{equation}
and
\begin{align}
\label{lem4}	\cm{A_\pm}{A_\mp^r}&=\pm A_\mp^{r-1}\qn{r}\qn{2B\mp(r-1)},
		&&\text{if}\quad \cm{A_\pm}{A_\mp}=\pm\qn{2B}, \\[4pt]
\label{lem5}	\cm{A_\pm}{A_\mp^r}&=\mp A_\mp^{r-1}\qn{r}\qn{2B\mp(r-1)},
		&&\text{if}\quad \cm{A_\pm}{A_\mp}=\mp\qn{2B}.
	\end{align}
\end{Lem}
\begin{proof}
All the statements are easily proved by means of an induction procedure.

Let us begin with \eqref{lem1}. The base case is obvious for the base case $r=1$,
$$
BA_{\pm}=A_{\pm}B\pm A_{\pm}=A_{\pm}(B\pm I),
$$
where we have used the commutation relation. Assuming its validity for any $r\in\mathbb{N}$, then
$$
B^{r+1}A_{\pm}=BA_{\pm}(B+I)^{r}=(A_{\pm}B\pm A_{\pm})(B\pm I)^{r}=A_{\pm}(B\pm I)^{r+1}.
$$

For \eqref{lem2} we have, for $r=1$,
\begin{align*}
\qn{\nu B+\eta}A_{\pm}
	&=\sum_{n=0}^{\infty}\frac{\gamma^{n}}{(2n+1)!}\sum_{\ell=0}^{n}\binom{n}{\ell}\eta^{n-\ell}\nu^{\ell}B^{\ell}A_{\pm} \\
	&=A_{\pm}\sum_{n=0}^{\infty}\frac{\gamma^{n}}{(2n+1)!}\sum_{\ell=0}^{n}\binom{n}{\ell}\eta^{n-\ell}\nu^{\ell}(B\pm I)^{\ell} \\
	&=A_{\pm}\qn{\nu B+\eta\pm\nu}.
\end{align*}
Let us suppose it is true for any $r\in\mathbb{N}$, so
$$
\qn{\nu B+\eta}A_{\pm}^{r+1}
	=A_{\pm}^{r}\qn{\nu B+\eta\pm\nu r}A_{\pm}
	=A_{\pm}^{r+1}\qn{\nu B+\eta\pm\nu(r+1)},
$$
where in the third step we have used the statement for the base case.

Regarding \eqref{lem3}, the case $r=1$ is just the hypothesis of the theorem. Supposing it valid for any $r\in\mathbb{N}$, and using the commutation relation and again the identity \eqref{lem1}, we get
$$
\cm{B}{A_{\pm}^{r+1}}=(\pm rA_{\pm}^{r}+A_{\pm}^{r}B)A_{\pm}-A_{\pm}^{r+1}B
		=\pm rA_{\pm}^{r+1}+A_{\pm}^{r+1}(B\pm I)-A_{\pm}^{r+1}B
		=\pm (r+1)A_{\pm}^{r+1}.
$$

Finally, let us prove now \eqref{lem4}. The case $r=1$ is just again the hypothesis of the statement. Let us assume it is true for any $r\in\mathbb{N}$, then by means of the commutation relation and the identity \eqref{lem2} we obtain
\begin{align*}
[A_{\pm},A_{\mp}^{r+1}]
	&=(\pm A_{\mp}^{r-1}\qn{r}\qn{2B\mp(r-1)}+A_{\mp}^{r}A_{\pm})A_{\mp}-A_{\mp}^{r+1}A_{\pm} \\
	&=\pm A_{\mp}^{r}\qn{r}\qn{2B\mp(r+1)}+A_{\mp}^{r}A_{\pm})A_{\mp}-A_{\mp}^{r+1}A_{\pm} \\
	&=\pm A_{\mp}^{r}(\qn{r}\qn{2B\mp(r+1)}+\qn{2B}).
\end{align*}
By straightforward but cumbersome calculations one can check that
$$
\qn{r}\qn{2B\mp(r+1)}+\qn{2B}=\qn{r+1}\qn{2B\mp r},
$$
which finishes the proof. The procedure to prove \eqref{lem5} is similar, so we omit it.
\end{proof}


\section{The $su_q(2)$ algebra}\label{sec_suq2Algebra} 
In this section we will introduce some basic facts about the $su_q(2)$ algebra. The main related references are~\cite{koelink,smi1,smi2}. We will assume that $q\in(0,1)$. Here, and throughout the paper, the operators will be written in capital letters whereas their eigenvalues in lower case letters.
	
\subsection{Basic facts about the unitary representations of the $su_q(2)$ algebra}	
The $su_q(2)$ algebra is generated by the operators $J_+$, $J_-$, and $J_0$, which fulfill the relations
$$
\cm{J_0}{J_\pm}=\pm J_\pm, \qquad \cm{J_+}{J_-}=\qn{J_0}
$$
and the adjointness properties
$$
J_\pm^\dag=J_\mp, \qquad J_0^\dag=J_0.
$$
A basis of any irreducible and unitary representation of finite dimension $D^{j}$ is given by
$$
\ket{jm}, \qquad j=0,\frac{1}{2},1,\frac{3}{2},\ldots, \quad m=-j,-j+1,\ldots,j-1,j.
$$
The action of the generators on the basis is given by 
\begin{align}
		J_{0}\ket{jm}&=m\ket{jm} , \label{j0m}\\[2pt]
		J_{-}\ket{jm}&=\sqrt{\qn{j+m}\qn{j-m+1}}\ket{j\,m-1}, \nonumber \\[2pt]
		J_{+}\ket{jm}&=\sqrt{\qn{j-m}\qn{j+m+1}}\ket{j\,m+1}. \nonumber 
\end{align}
Hence the explicit form of matrix elements of the irreducible representation $D^{j}$ is determined by the equations
\begin{align*}
		\mel{jm'}{J_0}{jm}&=m\delta_{m'm}, \\[2pt]
		\mel{jm'}{J_-}{jm}&=\sqrt{\qn{j+m}\qn{j-m+1}}\delta_{m'm-1}, \\[2pt]
		\mel{jm'}{J_+}{jm}&=\sqrt{\qn{j-m}\qn{j+m+1}}\delta_{m'm+1}.
\end{align*}
By induction it is easy to see that 	
\begin{equation}
\label{J^r}	 
	J_-^r\ket{jm}=\sqrt{\frac{\qn{j+m}!\qn{j-m+r}!}{\qn{j-m}!\qn{j+m-r}!}}\ket{j\,m-r}
	\quad\text{and}\quad
	J_+^r\ket{jm}=\sqrt{\frac{\qn{j-m}!\qn{j+m+r}!}{\qn{j+m}!\qn{j-m-r}!}}\ket{j\,m+r}.
\end{equation}
In particular, the matrix elements for powers of the generators are
$$
	\mel{jm'}{J_-^r}{jm}=\sqrt{\frac{\qn{j+m}!\qn{j-m+r}!}{\qn{j-m}!\qn{j+m-r}!}}\delta_{m'm-r}
	\quad\text{and}\quad
	\mel{jm'}{J_+^r}{jm}=\sqrt{\frac{\qn{j-m}!\qn{j+m+r}!}{\qn{j+m}!\qn{j-m-r}!}}\delta_{m'm+r}.
$$
Using the previous identities it is clear that
$$
	J_-^{j-m}\ket{jj}=\sqrt{\frac{\qn{2j}!\qn{j-m}!}{\qn{j+m}!}}\ket{jm}
	\quad\text{and}\quad
	J_+^{j+m}\ket{j\,-j}=\sqrt{\frac{\qn{2j}!\qn{j+m}!}{\qn{j-m}!}}\ket{jm},
$$
so
$$
\ket{jm}=\sqrt{\frac{\qn{j-m}!}{\qn{2j}!\qn{j+m}!}}J_{+}^{j+m}\ket{j\,-j}
\quad\text{or}\quad
\ket{jm}=\sqrt{\frac{\qn{j+m}!}{\qn{2j}!\qn{j-m}!}}J_{-}^{j-m}\ket{jj}.
$$
Next we define the Casimir operator of second order for $su_q(2)$ algebra by
$$
C=J_- J_+ +\qn{J_0+1/2}^2=J_{+}J_{-}+\qn{J_0-1/2}^2,
$$
which is is self-adjoint, $C^\dag=C$, due to $\mel{j'm'}{C}{jm}=\mel{jm}{C}{j'm'}.$

A fundamental property of the Casimir operator is its commutativity with the operators $J_0$ and $J_\pm$. A straightforward calculations show that 
$$
\cm{C}{J_0}=0 \quad\text{and}\quad \cm{C}{J_\pm}=0.
$$
Now, taking into account that 
$$
C\ket{jm}=\qn{j+1/2}^2\ket{jm}, 
$$
i.e., all vectors of the family $\ket{jm}$ are eigenvectors of the Casimir operator with the common eigenvalue $\qn{j+1/2}^2$ (which is the eigenvalue of the maximum weight vector $\ket{jj}$). Eq.~\eqref{j0m} together with the fact $\cm{C}{J_0}=0$ imply that operators $C$ and $J_0$ share a common orthonormal system of eigenvectors which is the family $\ket{jm}$ itself 
$$
\braket{j'm'}{jm}=\delta_{j'j}\delta_{m'm}.
$$


\subsection{Projection operator for $su_q(2)$ algebra}
Let be any vector $\ket{j'm'}$ of the basis of an unitary irreducible representation $D^{j'}$ such that $m'=j\geq0$. We define the application
\begin{equation}\label{def-po-su2} 
\begin{aligned}
		\proj{j}{jj}: & \qquad D^{j'} & \to  \quad & \text{span}\{\ket{jj}\} \\
		&  \ket{j'\, m'=j} & \mapsto \quad & \ \ \delta_{j'j}\ket{jj}.
	\end{aligned}
\end{equation}
Since we have chosen $m'=j\geq0$, there will exist an unitary irreducible representation $D^{j}\subseteq D^{j'}$. The action is the following one: from all possible vectors $\ket{j'\, m'=j}$ of $D^{j'}$, the application $\proj{j}{jj}$ extracts the maximum weight vector $\ket{jj}$ of the representation $D^{j}$.

First, we note that the projection operator commutes with the generator $J_{0}$ when it is applied to a vector $\ket{j'j}$, i.e., $\cm{\proj{j}{jj}}{J_{0}}\ket{j'j}=0$, so if we want to write the projection operator as an expansion in terms of the generators $J_-$ and $J_+$, they have to be raised to the same power, i.e.,
$$
	\proj{j}{jj}=\sum_{r=0}^\infty c_rJ_{-}^rJ_{+}^r.
$$
Actually previous expansion is not a series but a finite sum because if we apply the projection operator to a vector $\ket{j'j}$ then there will be an index $r_0$ such that $J_{+}^{r_0}\ket{j'j}=0$.

Let us now obtain an explicit expression for the coefficients $c_r$. It is clear by \eqref{def-po-su2} that
$$
\proj{j}{jj}\ket{jj}=
\begin{cases}
	\ket{jj}, \\
	c_0\ket{jj}+c_1J_{-}J_{+}\ket{jj}+\cdots=c_0\ket{jj},
\end{cases}
\Rightarrow\ c_0=1.
$$
For the remainder coefficients we note that $J_{+}\proj{j+}{jj}\ket{j'j}=0$ and, on the other hand,
by the equations \eqref{lem4} and \eqref{lem2} of Lemma~\ref{lem_aux1}
\begin{align*}
0=J_{+}\proj{j+}{jj}\ket{j'j}
	&=\sum_{r=0}^\infty c_rJ_{+}J_{-}^rJ_{+}^r\ket{j'j} \\
	&=\sum_{r=0}^\infty c_r\left(J_{-}^rJ_{+}^{r+1}+J_{-}^{r-1}\qn{r}\qn{2J_{0}-r+1}J_{+}^r\right)\ket{j'j} \\
	&=\sum_{r=0}^\infty\left(c_r+c_{r+1}\qn{r+1}\qn{2J_{0}+r}\right)J_{-}^rJ_{+}^{r+1}\ket{j'j}.
\end{align*}
Therefore, by \eqref{J^r} we get 
\begin{align*}
0
	&=\sum_{r=0}^\infty\left(c_r+c_{r+1}\qn{r+1}\qn{2J_{0}+r}\right)J_{-}^rJ_{+}^{r+1}\ket{j'j} \\
	&=\frac{\qn{j'-j-1}!}{\qn{j'+j}!}\sqrt{\frac{\qn{j'-j}}{\qn{j'+j+1}}}
		\sum_{r=0}^\infty\frac{\qn{j'+j+r+1}!}{\qn{j'-j-r-1}!}\left(c_r+c_{r+1}\qn{r+1}\qn{2j+r+2}\right)\ket{j'\,j+1}
\end{align*}
If we take $j'=j+a$, with $a\in\mathbb{N}$, we obtain on the right hand side
$$
\frac{\qn{a-1}!}{\qn{2j+a}!}\sqrt{\frac{\qn{a}}{\qn{2j+a+1}}}
\sum_{r=0}^{\infty}\frac{\qn{2j+r+1+a}!}{\qn{a-1-r}!}\left(c_r+c_{r+1}\qn{r+1}\qn{2j+r+2}\right)\ket{j+a\ j+1}.
$$
We note that the factor
$$
\frac{\qn{a-1}!}{\qn{a-1-r}!}=\qn{1}\qn{2}\cdots\qn{a-r-1}\qn{a-r}
$$
vanishes if $r\geq a$, so actually the series is a sum from $r=0$ to $r=a-1$, i.e., previous identity is
$$
\frac{\qn{a-1}!}{\qn{2j+a}!}\sqrt{\frac{\qn{a}}{\qn{2j+a+1}}}
	\sum_{r=0}^{a-1}\frac{\qn{2j+r+1+a}!}{\qn{a-r-1}!}\left(c_r+c_{r+1}\qn{r+1}\qn{2j+r+2}\right)\ket{j+a\ j+1}
=0,
$$
which implies
$$
\sum_{r=0}^{a-1}\frac{\qn{2j+r+1+a}!}{\qn{a-r-1}!}\left(c_r+c_{r+1}\qn{r+1}\qn{2j+r+2}\right)=0.
$$
For $a=1$ we have
$$
\qn{2j+2}!\left(c_0+c_1\qn{1}\qn{2j+2}\right)=0 \iff c_1=\frac{-c_0}{\qn{1}\qn{2j+2}}=-\frac{1}{\qn{1}\qn{2j+2}},
$$
and, in general, 
$$
\qn{2(j+n)}!\left(c_{n-1}+c_n\qn{n}\qn{2j+n+1}\right)=0 \iff c_n=(-1)^n\frac{\qn{2j+1}!}{\qn{n}!\qn{2j+n+1)}!},
$$
that can be checked by induction. Therefore, 
$$
\proj{j}{jj}=\sum_{r=0}^{\infty}(-1)^r\frac{\qn{2j+1}!}{\qn{r}!\qn{2j+r+1}!}J_-^r J_+^r.
$$
It is not complicate to prove that $\mel{j'm'}{\proj{j}{jj}}{jm}=\mel{jm}{\proj{j}{jj}}{j'm'}$, 
so $(\proj{j}{jj})^\dag=\proj{j}{jj}$, that is, the projection operator $\proj{j}{jj}$ is self-adjoint.

Let us consider now any vector $\ket{j'm'}$ of the basis of the unitary irreducible representation $D^{j'}$ 
such that $m'\leq j$, for $j\geq0$. We define the application
$$
	\proj{j}{mm'}: D^{j'}  \to  \text{span}\{\ket{jm}\}
$$
such that
	\begin{equation}\label{su2_proj-gen}
		\proj{j}{mm'}\ket{j'm'}=\sqrt{\frac{\qn{j+m}!}{\qn{2j}!\qn{j-m}!}}J_{-}^{j-m}\proj{j}{jj}J_{+}^{j-m'}
		\sqrt{\frac{\qn{j+m'}!}{\qn{2j}!\qn{j-m'}!}}\ket{j'm'}.
\end{equation}
We note that in the case $m=m'=j$ we recover the aforementioned projection operator, so we can understand this operator as a generalized projection operator. Here we must indicate that it is possible to define the case related to $m'>j$ by means of the already known identity
$$
	\ket{j'm'}=\sqrt{\frac{\qn{j'-m'}!}{\qn{2j'}!\qn{j'+m'}!}}J_{+}^{j'+m'}\ket{j'\, -j'},
$$
which allows us to obtain
\begin{align*}
		\proj{j}{mm'}\ket{j'm'}
		&=\sqrt{\frac{\qn{j+m}!}{\qn{2j}!\qn{j-m}!}}J_{-}^{j-m}\proj{j}{jj}
		J_{+}^{j-m'}\sqrt{\frac{\qn{j+m'}!}{\qn{2j}!\qn{j-m'}!}}\ket{j'm'} \\
		&=\sqrt{\frac{\qn{j+m}!}{\qn{2j}!\qn{j-m}!}}J_{-}^{j-m}\proj{j}{jj}
		J_{+}^{j-m'}\sqrt{\frac{\qn{j+m'}!}{\qn{2j}!\qn{j-m'}!}}\sqrt{\frac{\qn{j'-m'}!}{\qn{2j'}!\qn{j'+m'}!}}
		J_{+}^{j'+m'}\ket{j'\, -j'} \\
		&=\sqrt{\frac{\qn{j+m}!}{\qn{2j}!\qn{j-m}!}}J_{-}^{j-m}\proj{j}{jj}
		J_{+}^{j+j'}\sqrt{\frac{\qn{j+m'}!\qn{j'-m'}!}{\qn{2j}!\qn{2j'}!\qn{j-m'}!\qn{j'+m'}!}}\ket{j\, -j'}.
\end{align*}

The generalized projection operator $\proj{j}{mm'}$ fulfills the property $(\proj{j}{mm'})^\dag=\proj{j}{m'm}$. Indeed, a direct calculation shows
\begin{equation*}
\mel{jm}{\proj{j}{mm'}}{j'm'}=\mel{j'm'}{\proj{j}{m'm}}{jm}.
\end{equation*}
Moreover, applying this generalized projection operator to the elements of the basis we have, after some cumbersome but straightforward calculations, that
\begin{equation*}
\proj{j}{mm'}\ket{j'm'}=\delta_{j'j}\ket{jm},\qquad m'=-j',-j'+1,\ldots,j'-1,j'.
\end{equation*}

Thanks to the above identity we are able to prove a very interesting property of the generalized projection operator $\proj{j}{mm'}$. Given any linear combination $\ket{\cdot\,m'}=\sum_{j'}a_{j'm'}\ket{j'm'}$ of vectors $\ket{j'm'}$,  
 for a given $m'$, $m'=-j',-j'+1,\ldots,j'-1,j'$, we have
$$
	\proj{j}{mm'}\ket{\cdot\,m'}=\proj{j}{mm'}\sum_{j'}a_{j'm'}\ket{j'm'}=\sum_{j'}a_{j'm'}\proj{j}{mm'}\ket{j'm'}
	=\sum_{j'}a_{j'm'}\delta_{jj'}\ket{jm}=a_{jm'}\ket{jm}.
$$
So, the generalized projection operator $\proj{j}{mm'}$ applied on an arbitrary linear combination $\sum_{j'}a_{j'm'}\ket{j'm'}$ is proportional to $\ket{jm}$. Finally, let us point out that 
\begin{equation}\label{poxpo}
\proj{j}{mm'}\proj{j'}{m'm''}=\delta_{jj'} \proj{j}{mm''}.
\end{equation}


\section{The Clebsch-Gordan coefficients}\label{sec_CGC}
Let us consider the direct product $D^{j_1}\otimes D^{j_2}$ of two representations $D^{j_1}$ and $D^{j_2}$.
In general one have that the direct product of two representations can be expressed as a direct sum 
of the irreducible representations, i.e., 
$$
D^{j_1}\otimes D^{j_2}=\bigoplus_{j}  D^{j},
$$
where the direct sum runs in a certain set of discrete values related with $j_1$ and $j_2$ as we will see later on (we will not consider the case of the continuous spectra).

Let  $\ket{j_1m_1}$ and $\ket{j_2m_2}$ the orthogonal basis vectors of $D^{j_1}$ and $D^{j_2}$, respectively, and let $\ket{j_1 j_2;jm}$ be basis vectors of $D^{j}$. Then, a typical situation is when one expand the vectors $\ket{j_1 j_2;jm}$ in the basis $\ket{j_1m_1}\ket{j_2m_2}$, that is,
\begin{equation}\label{jm-jimi}
\ket{j_1,j_2:jm}=\sum_{\substack{m_1,\ m_2 \\ m=m_1+m_2}}C_{jm}^{j_1m_1,j_2m_2}\ket{j_1m_1}\ket{j_2m_2}
=\sum_{\substack{m_1,\ m_2 \\ m=m_1+m_2}}\braket{j_1m_1,j_2m_2}{jm}\ket{j_1m_1}\ket{j_2m_2}.
\end{equation}
The coefficients $C_{jm}^{j_1m_1,j_2m_2}$ of the above expansion is called the Clebsch-Gordan coefficients (CGC) and are usually denoted by $\braket{j_1m_1,j_2m_2}{jm}$. Our aim here is to compute them. There are several ways of finding the CGC, see e.g. \cite{hou90,koor89,lie92}, but we will use the generalized projection operator described in the previous section because it is more straightforward than the other ones.

Notice that one also can write the vectors $\ket{j_1m_1}\ket{j_2m_2}$ in the basis $\ket{j_1,j_2:jm}$
$$
\ket{j_1m_1}\ket{j_2m_2}=\sum_{\substack{j',\ m' \\ m'=m_1+m_2 \\ j'=j'(j_1,j_2)}}\tilde{C}_{j',m'}^{j_1m_1',j_2m_2'}
\ket{j_1,j_2:j'm'}
$$
Since the basis vectors $\ket{j_1m_1}$ and $\ket{j_2m_2}$ can be always assume to be orthonormal then, from \eqref{jm-jimi}, one find that 
$$
\braket{j_1m_1,j_2m_2}{jm}=  \bra{j_1m_1}\bra{j_2m_2}\cdot \ket{j_1,j_2:jm}.
$$
Notice that one also can write the vectors $\ket{j_1m_1'}\ket{j_2m_2'}$ in the basis $\ket{j_1,j_2:j'm'}$
$$
\ket{j_1m_1'}\ket{j_2m_2'}=\sum_{\substack{j',\ m' \\ m'=m_1'+m_2' \\ j'=j'(j_1,j_2)}}\tilde{C}_{j',m'}^{j_1m_1',j_2m_2'}
\ket{j_1,j_2:j'm'}.
$$
But we know that applying the  generalized projection operator $\proj{j}{mm'}$ (for the $q$-algebras we are interested in) to a linear combination (in $j'$) of vectors $\ket{j_1,j_2:j'm'}$ a vector proportional to 
$\ket{j_1,j_2:jm}$ is obtained, thus
$$
\proj{j}{mm'}\ket{j_1m'_1}\ket{j_2m'_2}= a_{jm'}\ket{j_1,j_2:jm}
\iff
\ket{j_1,j_2:jm}=\frac{\proj{j}{mm'}\ket{j_1m'_1}\ket{j_2m'_2}}{\|\proj{j}{mm'}\ket{j_1m'_1}\ket{j_2m'_2}\|},
$$
where, since, in our case  $(\proj{j}{mm'})^\dag=\proj{j}{m'm}$ and using \eqref{poxpo} we find
\[\begin{split}
\|\proj{j}{mm'}\ket{j_1m'_1}\ket{j_2m'_2}\|^2 & =
\braket{\bra{ j_1m'_1}\bra{j_2m'_2}(\proj{j}{mm'})^\dag}{\proj{j}{mm'}\ket{j_1m'_1}\ket{j_2m'_2}} \\
& =\braket{\bra{ j_1m'_1}\bra{j_2m'_2}}{\proj{j}{m'm'}\ket{j_1m'_1}\ket{j_2m'_2}},
\end{split}
\]
if we want $\ket{j_1,j_2:jm}$ to be normalized to 1. 

From the above it follows that the Clebsch-Gordan coefficients can be found in terms of the projection operator
by the expression
\begin{equation}\label{eq_GenCG}
\braket{j_1m_1,j_2m_2}{jm}=
\frac{\braket{\bra{ j_1m_1}\bra{j_2m_2}}{\proj{j}{mm'}\ket{j_1m'_1}\ket{j_2m'_2}}}
{\|\proj{}{mm'}\ket{j_1m'_1}\ket{j_2m'_2}\|}.
\end{equation}
Is a matter of fact that both $m'_1$ and $m'_2$ are free parameters, so we can choose the most appropriate values in each case to make easier the computation of the Clebsch-Gordan coefficients.

Let us now apply all of this to the irreducible representations of the $q$-algebra $su_q(2)$.


\subsection{Clebsch-Gordan coefficients for the $su_q(2)$}
Before calculate the Clebsch-Gordan coefficients with the formula \eqref{eq_GenCG}, we will discuss some preliminary results (for more details see e.g.~\cite{smi1}). 

Let us consider two different irreducible and unitary representations of order $j_1$ and $j_2$, i.e,
$$
D^{j_i}, \qquad \ket{j_im_i}, \quad m_i=-j_i,-j_i+1,\ldots,j_i-1,j_i, \qquad i=1,2,
$$
which has dimension $2j_i+1$ each one and whose generators are $J_\pm(i)$ and $J_0(i)$. The direct product representation $D^{j_1}\otimes D^{j_2}$ has a basis  $\ket{j_1m_1}\ket{j_2m_2}$ of dimension $(2j_1+1)(2j_2+1)$ and it is generated by the operators
$$
J_0(12)=J_0\otimes 1+1\otimes J_0 \quad\text{and}\quad
J_\pm(12)=J_\pm\otimes q^{J_0} + q^{-J_0}\otimes J_\pm,
$$
which, for convenience, we will rewrite as
$$
J_0(12)=J_0(1)+J_0(2) \quad\text{and}\quad J_\pm(12)=J_\pm(1)q^{J_0(2)}+q^{-J_0(1)}J_\pm(2).
$$
Here the number $i=1,2$ indicate to which representation $D^{j_i}$ the operators $J_0(i),J_\pm(i)$ are acting. Notice that these operators satisfy $(J_0(12))^\dag=J_0$, $(J_\pm(12))^\dag=J_\mp(12)$,
$$
q^{aJ_0(i)}J_0(i)=J_0(i)q^{aJ_0(i)} \quad\text{and}\quad
q^{aJ_0(i)}J_\pm(i)=J_\pm(i)q^{a(J_0(i)\pm1)},
$$
and
$$
\cm{J_0(12)}{J_\pm(12)}=\pm J_\pm(12)
\quad\text{and}\quad
\cm{J_+(12)}{J_-(12)}=\qn{2J_0(12)},
$$
so we can apply Lemma \ref{lem_aux1} with the operators $J_-(12)$, $J_+(12)$ and $J_0(12)$. This will be an important fact in order to construct the corresponding projection operator~$P_{mm'}^j(12)$. 

We will also need to know the expression of the powers of $J_\pm(12)$ and $J_0(12)$. In fact, by induction it can be shown that for all $r\in\mathbb{N}$,
$$
J_0^r(12)\ket{j_1m_1}\ket{j_2m_2}=(m_1+m_2)^r\ket{j_1m_1}\ket{j_2m_2},
$$
as well as
\begin{equation}\label{eq2_BinExp}
	J_\pm^r(12)=\sum_{\ell=0}^{r}\frac{\qn{r}!}{\qn{\ell}!\qn{r-\ell}!}J_\pm^\ell(1)J_\pm^{r-\ell}(2)
	q^{\ell J_0(2)-(r-\ell)J_0(1)}.
\end{equation}
Following the form of the Casimir operator for a single representation, we define the Casimir operator~by
$$
C(12)=J_{-}(12)J_{+}(12)+\qn{J_0(12)+1/2}^2=J_{+}(12)J_{-}(12)-\qn{2J_0(12)}+\qn{J_0(12)-1/2}^2.
$$

Now we are ready for computing the CGC. For doing that we will use \eqref{eq_GenCG} and we will choose $m'_1=j_1$ and $m'_2=j-j_1$. With this choice $m'=m'_1+m'_2=j$ and then
\begin{equation}\label{ccg-su2}
\braket{j_1m_1,j_2m_2}{jm}
=\frac{\braket{\bra{ j_1m_1}\bra{j_2m_2}}{\proj{j}{mj}\ket{j_1j_1}\ket{j_2\,j-j_1}}}
{\|\proj{j}{mj}\ket{j_1j_1}\ket{j_2\,j-j_1}\|},
\end{equation}
where the generalized projection operator \eqref{su2_proj-gen} is given by
\begin{align}\label{help-su2-1}
\proj{j}{mj} &=\sqrt{\frac{\qn{j+m}!}{\qn{2j}!\qn{j-m}!}}J_{-}^{j-m}(12)\proj{j}{jj} =\sqrt{\frac{\qn{j+m}!}{\qn{2j}!\qn{j-m}!}}\sum_{r=0}^\infty\frac{(-1)^r\qn{2j+1}!}{\qn{r}!\qn{2j+r+1}!}
J_{-}^{r+j-m}(12)J_{+}^{r}(12).
\end{align}
From \eqref{ccg-su2} it is clear that $j_1-j_2\leq j\leq j_1+j_2$. But if we choose $m'_2=j_2$ and $m'_1=j-j_2$ in \eqref{eq_GenCG}, then we get  $j_2-j_1\leq j\leq j_1+j_2$ and from where it follows that
$$
D^{j_1}\otimes D^{j_2}=\bigoplus_{j=|j_1-j_2|}^{j_1+j_2} D^{j}.
$$

We will start computing the numerator of \eqref{ccg-su2}. By the binomial expansion \eqref{eq2_BinExp}
\begin{align*}
	J_{+}^{r}(12)\ket{j_1j_1}\ket{j_2\, j-j_1}	&
=\sum_{\ell=0}^{r}\frac{\qn{r}!}{\qn{\ell}!\qn{r-\ell}!}J_{+}^{\ell}(1)J_{+}^{r-\ell}(2)
	q^{\ell J_0(2)-(r-\ell)J_0(1)}\ket{j_1j_1}\ket{j_2\, j-j_1} \\
& =\sum_{\ell=0}^{r}\frac{\qn{r}!q^{\ell(j-j_1)-(r-\ell)j_1}}
{\qn{\ell}!\qn{r-\ell}!}\bigg( J_{+}^{\ell}(1)\ket{j_1j_1}\bigg)\bigg( J_{+}^{r-\ell}(2)\ket{j_2\, j-j_1}\bigg)\\
& 
 =\sqrt{\frac{\qn{j_2-j+j_1}!\qn{j_2+j-j_1+r}!}{\qn{j_2+j-j_1}!\qn{j_2-j+j_1-r}!}}
	q^{-rj_1}\ket{j_1j_1}\ket{j_2\, j-j_1+r}.
\end{align*}
The last equality follows from \eqref{J^r} and the fact that $\ket{j_1j_1}$ is the maximal weight vector
so $J_{+}^{\ell}(1)\ket{j_1j_1}=0$ for all $\ell>0$. Notice also that if $r>j_1+j_2-j$, then the above
expression vanishes.

In a similar fashion we compute 
\begin{align*}
	J_{-}^{r+j-m}&\begin{aligned}[t]
		(12)\ket{j_1j_1}\ket{j_2\, j-j_1+r}
		=\sum_{\ell=0}^{r+j-m}\frac{\qn{r+j-m}!}{\qn{\ell}!\qn{r+j-m-\ell}!}J_{-}^{\ell}(1)J_{-}^{r+j-m-\ell}(2)
		q^{\ell J_0(2)-(r+j-m-\ell)J_0(1)} \\
		{}\times \ket{j_1j_1}\ket{j_2\, j-j_1+r}
	\end{aligned} \\
	&=\begin{aligned}[t]
		\sum_{\ell=0}^{r+j-m}\frac{\qn{r+j-m}!}{\qn{\ell}!\qn{r+j-m-\ell}!}  &
		\sqrt{\frac{\qn{2j_1}!\qn{\ell}!}{\qn{2j_1-\ell}!}\frac{\qn{j_2+j-j_1+r}!\qn{j_2+j_1-m-\ell}!}{\qn{j_1+j_2-j-r}!
				\qn{j_2-j_1+m+\ell}!}} \\ & 
\qquad \qquad 	\times q^{\ell(j+r)-(r+j-m)j_1}	\ket{j_1\, j_1-\ell}\ket{j_2\, m-j_1+\ell}.
	\end{aligned}
\end{align*}
In fact, in the above expression all terms for $\ell> 2j_1$ and $\ell < j_1-j_2-m$ vanish. Then, for the numerator on \eqref{ccg-su2} we get  
\begin{multline*}
	\sqrt{\frac{\qn{j+m}!}{\qn{2j}!\qn{j-m}!}}\sum_{r=0}^{j_1+j_2-j}\frac{(-1)^r\qn{2j+1}!}{\qn{r}!\qn{2j+r+1}!}
	\sum_{\ell=\max(0,j_1-j_2-m)}^{\min(r+j-m,2j_1)}\frac{\qn{r+j-m}!}{\qn{\ell}!\qn{r+j-m-\ell}!}
	\sqrt{\frac{\qn{2j_1}!\qn{\ell}!}{\qn{2j_1-\ell}!}} \\
	\times\begin{aligned}[t]
		\sqrt{\frac{\qn{j_2+j-j_1+r}!\qn{j_2+j_1-m-\ell}!}{\qn{j_2-j+j_1-r}!\qn{j_2-j_1+m+\ell}!}}
		\sqrt{\frac{\qn{j_2-j+j_1}!\qn{j_2+j-j_1+r}!}{\qn{j_2+j-j_1}!\qn{j_2-j+j_1-r}!}} \\
		{}\times q^{\ell(j+r)-(r+j-m)j_1-rj_1}\braket{j_1m_1}{j_1\,j_1-\ell}\braket{j_2m_2}{j_2\,m-j_1+\ell}.
	\end{aligned}
\end{multline*}
However, by orthogonality, there is only one addend in the sum, which corresponds with the index $m_1=j_1-\ell \iff \ell=j_1-m_1$. This implies that $\max(0,j_1-j_2-m)\leq j_1-m_1\leq \min(r+j-m,2j_1)$, therefore, after rearranging the summand we get for the numerator of \eqref{ccg-su2} the expression
\begin{multline}\label{eq2_NumCGFinal}
	\sqrt{\qn{2j+1}\frac{\qn{2j+1}!\qn{2j_1}!\qn{j_2+j_1-j}!\qn{j+m}!\qn{j_2-m_2}!}{\qn{j+j_2-j_1}!
			\qn{j-m}!\qn{j_1+m_1}!\qn{j_1-m_1}!\qn{j_2+m_2}!}}q^{mj_1-m_1j} \\
	\times \sum_{r=\max(0,j_1-j+m_2)}^{j_1+j_2-j}\frac{(-1)^{r}\qn{j-m+r}!\qn{j+j_2-j_1+r}!}{\qn{r}!\qn{2j+1+r}!
		\qn{j-j_1-m_2+r}!\qn{j_1+j_2-j-r}!}
	q^{-(j_1+m_1)r}.
\end{multline}

Let us now compute the denominator of \eqref{ccg-su2}. It is clear that it is the same as the numerator but with parameters $m_1=j_1$, $m_2=j-j_1$, so $m=j$, so it is equal to
\begin{equation}\label{eq2_DenCGExp}
\sqrt{\frac{\qn{2j+1}!\qn{j_1+j_2-j}!}{\qn{j+j_2-j_1}!}}
\left(\sum_{r=0}^{j_1+j_2-j}\frac{(-1)^r\qn{j+j_2-j_1+r}!}{\qn{r}!\qn{2j+1+r}!\qn{j_1+j_2-j-r}!}q^{-2j_1r}\right)^{1/2}.
\end{equation}
To compute the last sum we can use the formula \eqref{summa1} that leads to the value 
\begin{equation}\label{eq2_DenCGFinal}
	\sqrt{\frac{\qn{2j+1}!\qn{2j_1}!}{\qn{j_2+j+j_1+1}!\qn{j+j_1-j_2}!}}q^{(j_2+j-j_1+1)(j_2-j+j_1)/2}.
\end{equation}
Putting equations \eqref{eq2_NumCGFinal} and \eqref{eq2_DenCGFinal} together, we obtain the following expression for the Clebsch-Gordan coefficients for the $su_q(2)$
\begin{equation}\label{cgc-su2-1}
	\begin{split}
\braket{j_1m_1,j_2m_2}{jm}=	\sqrt{\qn{2j+1}\frac{\qn{j_1+j_2+j+1}!\qn{j+j_1-j_2}!\qn{j_1+j_2-j}!\qn{j+m}!\qn{j_2-m_2}!}
{\qn{j+j_2-j_1}!\qn{j-m}!\qn{j_1+m_1}!\qn{j_1-m_1}!\qn{j_2+m_2}!}}\\
\times q^{mj_1-m_1j-(j_2+j-j_1+1)(j_2-j+j_1)/2}   \sum_{r=\max(0,j_1-j+m_2)}^{j_1+j_2-j}
\frac{(-1)^{r}\qn{j-m+r}!\qn{j+j_2-j_1+r}!\,\, 
q^{-(j_1+m_1)r}}{\qn{r}!\qn{2j+1+r}!\qn{j-j_1-m_2+r}!\qn{j_1+j_2-j-r}!}.
\end{split}
\end{equation}
If we make the change $r\to j_1+j_2-j-r$ we recover the expression in \cite{smi1}
\begin{equation}\label{cgc-su2-2}
\begin{split}
	\braket{j_1m_1,j_2m_2}{jm}_q=	\sqrt{\qn{2j+1}\frac{\qn{j_1+j_2+j+1}!\qn{j+j_1-j_2}!\qn{j_1+j_2-j}!\qn{j+m}!\qn{j_2-m_2}!}
		{\qn{j+j_2-j_1}!\qn{j-m}!\qn{j_1+m_1}!\qn{j_1-m_1}!\qn{j_2+m_2}!}}\\
	\times q^{j_1m_2-j_2m_1-\frac12(j_1+j_2-j)(j_1+j_2+j+1)} \hspace*{-1cm} 
	\sum_{r=0}^{\min(j_2-m_2,j_1+j_2-j)}\hspace*{-.5cm} \frac{(-1)^{j_1+j_2-j+r}\qn{j_1+j_2-m-r}!\qn{2j_2-r}!\,\, 
		q^{(j_1+m_1)r}}{\qn{r}!\qn{j+j_1+j_2+1-r}!\qn{j_2-m_2-r}!\qn{j_1+j_2-j-r}!}. 
\end{split}
\end{equation}


\subsection{The representation as a terminating ${_3}F_2$ $q$-hypergeometric function and its consequences}
We are interested in writing the last sum in terms of the symmetric terminating $q$-hypergeometric function ${_3}F_2(-n,a,b;d,e|q,z)$ given in \eqref{q-hip-def}, so we should fix the value of the nonnegative integer $n$. A direct calculation shows that we can fix $n=j_1+j_2-j$ or $n=j_2-m_2$ independently of which one is bigger. The main reason is that, independently of the choice we make, all the extra terms in the sum will vanish. Taking this into account, we can write the above expression \eqref{cgc-su2-2} as a symmetric terminating $q$-hypergeometric function (notice that $j_1+j_2+j+1$ is always bigger than $\max(j_2-m_2,j_1+j_2-j)$)
\begin{small}
\begin{equation*}\label{cgc-su2-3F2-long}
\begin{split}&
\braket{j_1m_1,j_2m_2}{jm}_q= 	\frac{(-1)^{j_1+j_2-j}  q^{j_1m_2-j_2m_1-\frac12(j_1+j_2-j)(j_1+j_2+j+1)}
[2j_2]_q!\qn{j_1+j_2-m}! }
{\sqrt{\qn{j_1+m_1}!\qn{j_1-m_1}!\qn{j_2+m_2}!\qn{j_2-m_2}!\qn{j-m}!}}
\\ & 
\times  \sqrt{\frac{\qn{2j+1}\qn{j+m}!\qn{j+j_1-j_2}!}{\qn{j_1+j_2+j+1}!\qn{j+j_2-j_1}!\qn{j_1+j_2-j}!}} \,
{_3}F_2 \bigg(\!\begin{array}{c}
  j-j_1-j_2\, , \, m_2-j_2\, ,\, -j-j_1-j_2-1 \\
  m-j_1-j_2 \, , \, -2j_2  
         \end{array}\!\bigg\vert\,  q\, , \, q^{j_1+m_1}\!\bigg).
\end{split}
\end{equation*}
\end{small}%
For convenience we will write the last formula as follows
\begin{equation}\label{cgc-su2-3F2}
	\braket{j_1m_1,j_2m_2}{jm}_q= (-1)^{j_1+j_2-j} \, \Gamma^{j_1,j_2,j}_{m_1,m_2,m}
	{_3}F_2 \bigg(\!\begin{array}{c}
		j-j_1-j_2\, , \, m_2-j_2\, ,\, -j-j_1-j_2-1 \\
		m-j_1-j_2 \, , \, -2j_2  
	\end{array}\!\bigg\vert\,  q\, , \, q^{j_1+m_1}\!\bigg),
\end{equation}
where
\begin{small}
	$$\Gamma^{j_1,j_2,j}_{m_1,m_2,m}=	\frac{ q^{j_1m_2-j_2m_1-\frac12(j_1+j_2-j)(j_1+j_2+j+1)}
		[2j_2]_q!\qn{j_1+j_2-m}!  \sqrt{\qn{2j+1}\qn{j+m}!\qn{j+j_1-j_2}!} }
	{\sqrt{\qn{j_1+j_2+j+1}!\qn{j+j_2-j_1}!\qn{j_1+j_2-j}!\qn{j_1+m_1}!\qn{j_1-m_1}!\qn{j_2+m_2}!\qn{j_2-m_2}!\qn{j-m}!}}.
	$$
\end{small}
In the last formula \eqref{cgc-su2-3F2}, as well as in \eqref{cgc-su2-2}, we explicitly write down the base $q$ we are using because it will be important when we discuss the symmetry properties of the CGC.


\subsubsection*{Symmetry properties}
Next we will obtain the symmetry properties for the CGC of the $su_q(2)$ algebra. In order to that, we will use the representation formula~\eqref{cgc-su2-3F2}.

If we use the transformation \eqref{142q} and the formulas \eqref{q-fac-poc} with the choice $n=j_1+j_2-j$, $a=-j-j_1-j_2-1$, $b=m_2-j_2$, $d=m-j_1-j_2$ and $e=-2j_2$, then we find, after a straightforward calculation the following symmetry formula
\begin{equation}\label{simj1j2}
	\braket{j_1m_1,j_2m_2}{jm}_q=(-1)^{j_1+j_2-j}\braket{j_2m_2,j_1m_1}{jm}_{1/q}.
\end{equation}
Similarly, if instead of  \eqref{142q} we use \eqref{141q} we find 
\begin{equation}\label{sim-m}
	\braket{j_1m_1,j_2m_2}{jm}_q=\braket{j_2\,-m_2,j_1\,-m_1}{j\,-m}_{q}=
	(-1)^{j_1+j_2-j}\braket{j_1\,-m_1,j_2\,-m_2}{j\,-m}_{1/q},
\end{equation}
where the last equality follows by using the symmetry~\eqref{simj1j2}.

In a similar way we now set $n=m_2-j_2$, $a=-j-j_1-j_2-1$, $b=j_1+j_2-j$,  $d=-2j_2$, and $e=m-j_1-j_2$. We use the transformation formula \eqref{142q}, with the minus sign, to obtain
\begin{equation}\label{sim-jj1}
	\braket{j_1m_1,j_2m_2}{jm}_q=(-1)^{j_2+m_2}q^{-m_2} \sqrt{\frac{\qn{2j+1}}{\qn{2j_1+1}}} \braket{j\,-m,j_2m_2}{j_1\,-m_1}_{1/q}.
\end{equation}
If in the right hand side of \eqref{sim-jj1} we apply consecutively the symmetries \eqref{simj1j2} and \eqref{sim-m} 
we obtain the following relation  
\begin{equation}\label{sim-jj1-smi}
	\braket{j_1m_1,j_2m_2}{jm}_q=(-1)^{j_2+m_2}q^{-m_2} \sqrt{\frac{\qn{2j+1}}{\qn{2j_1+1}}} \braket{j_2\,-m_2,j m}{j_1 m_1}_{1/q}.
\end{equation}
Combining properly all the above formulas we can find a lot of new symmetries. For example, we can rewrite 
the last formula \eqref{sim-jj1-smi} by interchanging the indexes $1\leftrightarrow2$ as well as $q\to1/q$ 
to get
\begin{equation}\label{sim-jj2-help}
\braket{j_2m_2,j_1m_1}{jm}_{1/q}=(-1)^{j_1+m_1}q^{m_1} \sqrt{\frac{\qn{2j+1}}{\qn{2j_2+1}}}
\braket{j_1\,-m_1,j m}{j_2 m_2}_{q},
\end{equation}
and then use \eqref{simj1j2} in the right and left sides of the obtained equation to get another 
symmetry property
\begin{equation}\label{sim-jj2-smi}
\braket{j_1m_1,j_2m_2}{jm}_q=(-1)^{j_1\,-m_1}q^{m_1}\sqrt{\frac{\qn{2j+1}}{\qn{2j_2+1}}} \braket{jm,j_1-m_1}{j_2 m_2}_{1/q}.
\end{equation}
Formulas \eqref{sim-jj1-smi} and \eqref{sim-jj2-smi} have been obtained in \cite{smi2} by a much more complicated 
method that the one used here.

If we now in \eqref{sim-jj2-help} make the change $m_1\to-m_1$, $m_2\to-m_2$, $m\to-m$, $q\to1/q$ and apply to the obtained CGC on the left hand side of the first equality in \eqref{sim-m}, we find another 
relevant symmetry~\cite{smi2}
\begin{equation}\label{sim-jj2m-rose}
\braket{j_1m_1,j_2m_2}{jm}_q=(-1)^{j_1-m_1}q^{m_1}\sqrt{\frac{\qn{2j+1}}{\qn{2j_2+1}}} \braket{j_1m_1,j\,-m}{j_2\,-m_2}_{1/q}.
\end{equation}

It can be also established, by a direct calculation, that the right hand side of \eqref{cgc-su2-2} or \eqref{cgc-su2-3F2} is invariant under the change $j_1\to \frac{j_1+j_2+m_1+m_2 }2$, $m_1\to \frac{j_1-j_2+m_1-m_2 }2$, $j_2\to \frac{j_1+j_2-m_1-m_2 }2$, $m_2\to \frac{j_1-j_2-m_1+m_2 }2$, $j\to j$, $m\to j_1-j_2$, so
\begin{equation}\label{sim-regge}
\begin{array}{c}
	\braket{j_1m_1,j_2m_2}{jm}_q=	\braket{\frac{j_1+j_2+m_1+m_2 }2 \,\frac{j_1-j_2+m_1-m_2 }2,
		\frac{j_1+j_2-m_1-m_2 }2 , 	\frac{j_1-j_2-m_1+m_2 }2 }{j\, j_1-j_2}_q,
\end{array}
\end{equation}
which is one of the 72 symmetry properties for the CGC discovered by Regge for the case $q\to1$ \cite{reg58} (for the $q$-case see \cite[\S4.3]{smi2}). The other 71 can be obtained by combining the above symmetry \eqref{sim-regge} with \eqref{simj1j2}, \eqref{sim-m}, and \eqref{sim-jj1}.

Finally, let us point out that all the obtained symmetries becomes into the classical ones \cite[\S8.4.3]{var:88} when $q\to1$.


\subsubsection*{Another explicit formula for the CGC}
Let us apply \eqref{142q} to the ${_3}F_2$ in \eqref{cgc-su2-3F2} setting $n=j_1+j_2-j$, $a=m_2-j_2$, $b=-j-j_1-j_2-j-1$, $d=m-j_1-j_2$, and $e=-2j_2$. This yields to the equivalent representation formula
\begin{multline}\label{cgc-su2-3F2-RW1}
\braket{j_1m_1,j_2m_2}{jm}_q= \frac{\sqrt{\qn{2j+1}\qn{j_1-m_1}! \qn{j_2+m_2}! \qn{j-m}!\qn{j+m}!\qn{j+j_1-j_2}!\qn{j+j_2-j_1}!}}
{\sqrt{\qn{j_1+m_1}!\qn{j_2-m_2}!\qn{j_1+j_2+j+1}!\qn{j_1+j_2-j}!}} \\  
\times \frac{(-1)^{j_1+j_2-j}  q^{j_1m_2-j_2m_1-\frac12(j_1+j_2-j)(j_1+j_2+j+1)}}{\qn{j-j_2-m_1}! \qn{j-j_1+m_2}! }
{_3}F_2 \bigg(\!\begin{array}{c}
	j-j_1-j_2\, , \, m_2-j_2\, ,\, -m_1-j_1 \\
	j-j_2-m_1+1 \, , \, j-j_1+m_2+1 \end{array}\!\bigg\vert\, q\, , \, q^{j_1+j_2+j+1}\!\bigg).
\end{multline}
If in the above relation we interchange the indexes $1$ and $2$, change $q\to 1/q$ and use the symmetry relation \eqref{simj1j2} we get the representation 
\begin{multline}\label{cgc-su2-3F2-RW2}
\braket{j_1m_1,j_2m_2}{jm}_q= 	\frac{\sqrt{\qn{2j+1}\qn{j_1+m_1}! \qn{j_2-m_2}! \qn{j-m}!\qn{j+m}!\qn{j+j_1-j_2}!\qn{j+j_2-j_1}!}}
{\sqrt{\qn{j_1-m_1}!\qn{j_2+m_2}!\qn{j_1+j_2+j+1}!\qn{j_1+j_2-j}!}} \\  
\times\frac{ q^{j_1m_2-j_2m_1+\frac12(j_1+j_2-j)(j_1+j_2+j+1)} }{\qn{j-j_2+m_1}! \qn{j-j_1-m_2}! }
{_3}F_2 \bigg(\!\begin{array}{c}
	j-j_1-j_2\, , \, m_1-j_1\, ,\, -m_2-j_2 \\
	j-j_2+m_1+1 \, , \, j-j_1-m_2+1 \end{array}\!\bigg| q\, , \, q^{-j_1-j_2-j-1}\!\bigg),
\end{multline}
from where, by using \eqref{q-fac-poc}, we obtain
\begin{small}
\begin{equation}\label{cgc-su2-RW-for}
\begin{split}
& \braket{j_1m_1,j_2m_2}{jm}_q=	q^{j_1m_2-j_2m_1-\frac12(j_1+j_2-j)(j_1+j_2+j+1)} \sqrt{ \qn{2j+1} }  \\ &  	\times	
\sqrt{ \frac{\qn{j_1+m_1}!\qn{j_1-m_1}!\qn{j_2+m_2}!\qn{j_2-m_2}!\qn{j+m}!\qn{j-m}!\qn{j_1+j_2-j}!\qn{j+j_1-j_2}!\qn{j+j_2-j_1}!} {\qn{j_1+j_2+j+1}!}} \\ & \qquad 	\times	
\sum_{r=0}^{\infty}  \frac{(-1)^{r}  q^{-(j_1+j_2+j+1)r }}
{\qn{r}!\qn{j_1+j_2-j-r}!\qn{j_2+m_2-r}!\qn{j_1-m_1-r}!\qn{j-j_2+m_1+r}!\qn{j-j_1-m_2+r}!},
\end{split}
\end{equation}
\end{small}%
which is the $q$-analogue of the the Racah formula for the CGC of the $su_q(2)$ algebra \cite[Eq.~(3), p.~238]{var:88}. Here, the summation index are the integers for which all the factorial arguments and nonnegative. 

The last formula can be written in terms of the $q$-analogue of the binomial coefficients \eqref{q-bin-coe}, so
\begin{small}
\begin{equation}\label{cgc-su2-RW-for-bino}
\begin{split}
\braket{j_1m_1,j_2m_2}{jm}_q=& 	q^{j_1m_2-j_2m_1-\frac12(j_1+j_2-j)(j_1+j_2+j+1)}  \sqrt{
	\frac{\qbcs{2j_1}{j_1+j_2-j} \qbcs{2j_2}{j_1+j_2-j}}
	{\qbcs{j_1+j_2+j+1}{j_1+j_2-j} \qbcs{2j_1}{j_1-m_1} \qbcs{2j_2}{j_2-m_2} \qbcs{2j}{j-m}}
}  \\ &   	\times	
\sum_{r=0}^{\infty}   (-1)^{r}  
\qbc{j_1+j_2-j}{r}  \qbc{j+j_1-j_2}{j_1-m_1-r}  \qbc{j+j_2-j_1}{j_2+m_2-r} q^{-(j_1+j_2+j+1)r},
\end{split}
\end{equation}
\end{small}%
which is seem to be, at least for the limit case $q\to1$, very convenient to compute the explicit values of the CGC \cite{shi60}. Let us also point out that the symmetry formulas \eqref{sim-m}, \eqref{simj1j2}, and \eqref{sim-jj2m-rose} can be directly obtained from the $q$-analogue of the Racah formula \eqref{cgc-su2-RW-for} or \eqref{cgc-su2-RW-for-bino} in a similar way as it was done for the limit case $q\to1$ in~\cite[p.~40--41]{rose}.

To conclude this series of CGC formulas, we obtain here another useful hypergeometric representation. By means of the symmetry \eqref{simj1j2} in the left hand side of \eqref{sim-jj2-help}, and then using the representation \eqref{cgc-su2-3F2} to the resulting CGC on right hand side we get the expression
\begin{small}
\begin{multline}\label{cgc-su2-3F2-long-equiv}
	\braket{j_1m_1,j_2m_2}{jm}_q=
	\frac{(-1)^{j_1-m_1}  q^{j_1m+j m_1+m_1-\frac{1}{2}(j_1+j-j_2)(j_1+j_2+j+1)}\qn{2j}!\qn{j_1+j-m_2}!}{\sqrt{\qn{j_1+m_1}!\qn{j_1-m_1}!\qn{j+m}!\qn{j_2-m_2}!\qn{j-m}!}}
			\\ 
	\times \sqrt{\frac{\qn{2j+1}\qn{j_2+m_2}!\qn{j_1+j_2-j}!}{\qn{j_1+j_2+j+1}!\qn{j+j_2-j_1}!\qn{j_1+j-j_2}!}} \,	{_3}F_2 \bigg(\!\begin{array}{c}
				j_2-j_1-j\, , \, m-j\, ,\, -j-j_1-j_2-1 \\
				m_2-j_1-j \, , \, -2j  
\end{array}\!\bigg\vert\, q\, , \, q^{j_1-m_1}\!\bigg).
\end{multline}
\end{small}


\subsubsection*{Some special values}
From the representation of the CGC in terms of the ${_3}F_2$ we can obtain some relevant values.

First one is when $j$ reaches its maximal value, i.e., $j=j_1+j_2$. In this case the ${}_3F_2$ function in \eqref{cgc-su2-3F2} is equal to one, so 
\begin{equation}\label{cgc-j=j1+j2}
\braket{j_1m_1,j_2m_2}{j_1+j_2\,m}=
q^{j_1m_2-j_2m_1} 
\sqrt{\frac{\qn{2j_1}!\qn{2j_2}!\qn{j_1+j_2+m}!\qn{j_1+j_2-m}!}
{\qn{2j_1+2j_2}!\qn{j_1+m_1}!\qn{j_1-m_1}!\qn{j_2+m_2}!\qn{j_2-m_2}!}}.
\end{equation}
Another useful value is $j=j_1+j_2-1$. In this case the hypergeometric function reduces to a sum of two terms (the first parameter in \eqref{q-hip-def} is $a_1=-1$), so we find, after some simplifications,
\begin{multline}\label{cgc-j=j1+j2-1}
\braket{j_1m_1,j_2m_2}{j_1+j_2-1\, m}=  
q^{j_1m_2-j_2m_1-j_1-j_2} \left\{ \qn{2j_1-2j_2}\qn{j_2-m_2} q^{j_1+m_1}-\qn{2j_2}\qn{j_1+j_2-m}  \right\}
\\
\times \sqrt{\frac{\qn{2j_1+2j_2-1}\qn{2j_1-1}!\qn{2j_2-1}!\qn{j_1+j_2+m-1}!\qn{j_1+j_2-m-1}!}
{\qn{2j_1+2j_2}!\qn{j_1+m_1}!\qn{j_1-m_1}!\qn{j_2+m_2}!\qn{j_2-m_2}!}}.
\end{multline}

If now set $j=j_1-j_2$ (we assume without loss of generality that $j_1\geq j_2$), i.e., the case when $j$ reaches its minimal value, then using \eqref{suma2} with $n=j_2-m_2$, $b=2j_1+1$, $c=j_1+j_2-m$, we find
\begin{multline}\label{cgc-j=j1-j2}
\braket{j_1m_1,j_2m_2}{j_1-j_2\, m} =
(-1)^{j_2+m_2}	q^{-j_1m_2-j_2m_1-m_2}
\\
\times \sqrt{\frac{\qn{2j_1-2j_2+1}!\qn{2j_2}!\qn{j_1+m_1}!\qn{j_1-m_1}!}
{\qn{2j_1+1}\qn{j_1-j_2-m}!\qn{j_1-j_2+m}!\qn{j_2+m_2}!\qn{j_2-m_2}!}}.
\end{multline}

If we put $m=j$ in \eqref{cgc-su2-3F2} and use the formula \eqref{suma2} with $n=j_2-m_2$, $b=j_1+j_2+j+1$, $c=2j_2$, we get
\begin{multline}\label{cgc-m=j}
\braket{j_1m_1,j_2m_2}{jj} =
(-1)^{j_1-m_1} q^{\left(j_{1}+j_{2}-j\right)\left(j+j_{2}-j_{1}+1\right)/2-(j+1)(j_1-m_1)}  \\
 \times \sqrt{\frac{\qn{2j+1}!\qn{j_1+m_1}!\qn{j_2+m_2}\qn{j_1+j_2-j}!}
{\qn{j_1-j_2+j}!\qn{j_2-j_1+j}!\qn{j_1+j_2+j+1}!\qn{j_1-m_1}!\qn{j_2-m_2}!}}.
\end{multline}
The same value is obtained if we use the expression~\eqref{cgc-su2-3F2-long-equiv}.

Putting $j=0$ in the above formula \eqref{cgc-m=j} and taking into account that, in this case, $j_1=j_2$, $m_1+m_2=0$, we obtain the value
$$
\braket{j_1m_1,j_1\,-m_1}{00}=\delta_{j_1j_2}\delta_{m_1\,-m_2} \frac{(-1)^{j_1-m_1}q^{m_1}}{\sqrt{[2j_1+1]_q}}.
$$
Notice also that, when $j=j_1+j_2$ and $m=j_1+j_2$ ($m_1=j_1$, $m_2=j_2$), then $\braket{j_1m_1,j_2m_2}{j_1+j_2\,j_1+j_2}=1$, which is the standard normalization for the CGC. 

Notice from the formula \eqref{cgc-su2-3F2-long-equiv} the value of $\braket{j_1m_1,j_2m_2}{j\,j-1}$ can be easily obtained since, in this case of equation \eqref{cgc-j=j1+j2-1}, one of the parameter of the hypergeometric function is $-1$ so it reduces to a sum of two terms.

The ${}_3F_2$ function in \eqref{cgc-su2-3F2} is equal to one also when $m_2=j_2$, thus we find
\begin{multline}\label{cgc-m2=j2}
\braket{j_1m_1,j_2j_2}{jm} =
(-1)^{j_1+j_2-j} q^{j_2(j_1-m_1)-\frac12(j_1+j_2-j)(j_1+j_2+j+1)} 
\\ 
\times \sqrt{\frac{\qn{2j+1}\, \qn{j+j_1-j_2}!\qn{j+m}!\qn{j_1-m_1}! \qn{2j_2}!}
{\qn{j_1+j_2+j+1}!\qn{j_1+j_2-j}!\qn{j+j_2-j_1}!\qn{j-m}!\qn{j_1+m_1}!}}.
\end{multline}
Notice that from \eqref{cgc-su2-3F2} the value of the $\braket{j_1m_1,j_2\,j_2-1}{jm}$ can be also obtained (again the ${_3}F_2$ reduces to a sum of two terms).

If we now put $m_1=j_1$ in \eqref{cgc-su2-3F2} and use the formula \eqref{suma2} with $n=j_1+j_2-j$, $b=j_1+j_2+j+1$ and $c=2j_2$, we get\footnote{Here we correct a misprint in~\cite{smi1}.}
\begin{multline}\label{cgc-m1=j1}
\braket{j_1j_1,j_2m_2}{jm}=
q^{-j_1(j_2-m_2)+\frac12(j_1+j_2-j)(j_1+j_2+j+1)}
\\ 
\times \sqrt{\frac{\qn{2j+1}\, \qn{j+j_2-j_1}!\qn{j+m}!\qn{j_2-m_2}! \qn{2j_1}!}
	{\qn{j_1+j_2+j+1}!\qn{j_1+j_2-j}!\qn{j+j_1-j_2}!\qn{j-m}!\qn{j_2+m_2}!}}.
\end{multline}

Using the formulas \eqref{cgc-su2-3F2-RW1} and \eqref{cgc-su2-3F2-RW2} we find the following expressions for the special values of the CGC when $m_1=-j_1$ and $m_2=-j_2$, respectively,
\begin{multline}\label{cgc-m1=-j1}
\braket{j_1-j_1,j_2m_2}{jm}=(-1)^{j_1+j_2-j}q^{j_1(j_2+m_2)+\frac12(j_1+j_2-j)(j_1+j_2+j+1)}
\\ 
\times  \sqrt{ \frac{ \qn{2j+1}\,\qn{2j_1}! \qn{j_2+m_2}! \qn{j-m}! \qn{j+j_2-j_1}!}
{\qn{j_1+j_2+j+1}!\qn{j+j_1-j_2}!\qn{j_1+j_2-j}!\qn{j+m}!\qn{j_2-m_2}!}}.
 \end{multline}
and
\begin{multline}\label{cgc-m1=-j2}
\braket{j_1m_1,j_2\,-j_2}{jm}  = q^{-j_2(j_1+m_1)-\frac12(j_1+j_2-j)(j_1+j_2+j+1)}
\\ 
 \times  \sqrt{ \frac{ \qn{2j+1}\,\qn{2j_2}! \qn{j_1+m_1}! \qn{j-m}! \qn{j+j_1-j_2}!}
{\qn{j_1+j_2+j+1}!\qn{j+j_2-j_1}!\qn{j_1+j_2-j}!\qn{j+m}!\qn{j_1-m_1}!}}.
\end{multline}

Finally, we will consider another important case: $m=m_1=m_2=0$. Notice that, in this situation
one should have $j_1$ and $j_2$ integers. In this case is convenient to use \eqref{cgc-su2-3F2-RW1}
that leads to the value
\begin{multline}\label{cgc-su2-3F2-m=0}
\braket{j_1 0,j_2 0}{j 0}_q=   \sqrt{\frac{\qn{2j+1}  \qn{j+j_1-j_2}!\qn{j+j_2-j_1}!}{\qn{j_1+j_2+j+1}!\qn{j_1+j_2-j}!}} q^{-\frac12(j_1+j_2-j)(j_1+j_2+j+1)}
\\  
\times \frac{(-1)^{j_1+j_2-j} 	\qn{j}!}{\qn{j-j_2}! \qn{j-j_1}!}
{_3}F_2 \bigg(\!\begin{array}{c}
j-j_1-j_2\, , \, -j_1\, ,\, -j_2 \\
j-j_1+1 \, , \, j-j_2+1 \end{array}\!\bigg\vert\, q\, , \, q^{j_1+j_2+j+1}\!\bigg).
\end{multline}
This formula can not be summed up as far as we know, but if one takes the limit when $q\to1$ we find for the ${_3}F_2$ function the value
$$
{_3}F_2 \bigg(\!\begin{array}{c}
j-j_1-j_2\, , \, -j_1\, ,\, -j_2 \\
j-j_1+1 \, , \, j-j_2+1 \end{array}\!\bigg\vert\, 1\!\bigg).
$$
Taking into account that $j-j_1-j_2$ is a negative integer, then we can use Dixon's Theorem \cite[Theorem 3.4.1, p.~143]{aar} to obtain the value
\begin{equation}\label{cgc-su2-3F2-m=0-q=1}
\begin{small}
\braket{j_1 0,j_2 0}{j 0}= 
\begin{cases} 
\dfrac{(-1)^{k-j}k! }{(k-j_1)!(k-j_2)!(k-j)!} \sqrt{\dfrac{(2j+1)(2k-2j_1)!(2k-2j_2)! (2k-2j)!}{(2k+1)!}}, & 
\text{if $j_1+j_2+j=2k$,} \\
0, & \text{if $j_1+j_2+j=2k+1$},
\end{cases}
\end{small}
\end{equation}
where $k=1,2,3,\dots$; it coincides with the one given in \cite[Eq.~(32), p.~251]{var:88}. In this case ($q=1$) the zero value for the case when  $j_1+j_2+j$ is an odd positive integer immediately follows from the symmetry property \eqref{sim-m}, something that does not happen for the $q$ case owing to the change of $q$ by $1/q$ in the right side of \eqref{sim-m}. In fact, in the tables obtained in \cite{smi1} it can be seen some values of the CGC for $m_1=m_2=m=0$ that are not zero when  $j_1+j_2+j$ is an odd positive integer, but becomes zero in the limit $q\to0$.

Before going ahead, let us point out that all the formulas from \eqref{cgc-j=j1+j2}--\eqref{cgc-m1=-j2} become into the classical ones \cite[\S8.5]{var:88} by taking the limit $q\to1$.


\subsubsection*{Connection with the $q$-Hahn polynomials}
As it is shown\footnote{The interested reader in the theory of orthogonal polynomials on nonuniform lattices developed by Nikiforov and Uvarov is urged to consult~\cite{ran,nsu}.} in \cite{nsu} (see also \cite{ran}), there are several analogues of the
classical Hahn polynomials but, among them, we will use the one introduced in \cite[Eq.~(3.11.53), p.~150]{nsu} and depply studied in \cite{arv} (for the case of the original $q$-Hahn polynomials see e.g. \cite{koor89,vil:92}). The main reason is related to the fact that for this family the limit $q\to1$ leads directly, without introducing any rescaling factor, to the corresponding formula for the classical ($q=1$) Hahn polynomials.

Our starting point will be different from the one used in \cite{arv} or the one used for the classical case 
in \cite[\S5.2.2.3 p.~245--246]{nsu}. Our idea is to exploit the representation of the CGC in terms of the
$q$-hypergeometric function ${_3}F_2$ given in \eqref{cgc-su2-3F2}. So, we start rewriting that hypergeometric representation~as
\begin{equation}\label{cgc-su2-3F2-short}
\braket{j_1m_1,j_2m_2}{jm}_q= (-1)^{j_1+j_2-j}\, \Gamma^{j_1,j_2,j}_{m_1,m_2,m}
 	{_3}F_2 \bigg(\!\begin{array}{c}
		m_2-j_2	\, , \,  -j-j_1-j_2-1  \, , \,  	j-j_1-j_2   \\
		m-j_1-j_2 \, , \, -2j_2  
		\end{array}\!\bigg\vert\, q\, , \, q^{j_1+m_1}\!\bigg),
\end{equation}
and using the transformation \eqref{142q} with $n=j_2-m_2$, $a= -j-j_1-j_2-1$, $b=j-j_1-j_2$, $d=m-j_1-j_2$ and $e=-2j_2$ to get 
$$
\braket{j_1m_1,j_2m_2}{jm}_q= (-1)^{j_1+j_2-j} \, \widetilde{\Gamma}^{j_1,j_2,j}_{m,m_1,m_2}
{_3}F_2 \bigg(\!\begin{array}{c}
	m_2-j_2	\, , \,  -j-j_1-j_2-1  \, , \,  m-j   \\
	m-j_1-j_2 \, , \, m_2-j-j_1  
\end{array}\!\bigg\vert\, q\, , \, q^{-(j+j_2-j_1)}\!\bigg),
$$
where 
$$
\widetilde{\Gamma}^{j_1,j_2,j}_{m,m_1,m_2}=q^{(j_2-m_2)(j+j_1+j_2+1)}\frac{(j+j_1-j_2+1|q)_{j_2-m_2}}{(-2j_2|q)_{j_2-m2}} \Gamma^{j_1,j_2,j}_{m_1,m_2,m}.
$$
Next, we apply to the last expression for the CGC the transformation \eqref{142q} again, but this time with parameters $n=j_2-m_2$, $a=m-j$, $b=-j-j_1-j_2-1$, $d=m-j_1-j_2$ and $e=m_2-j-j_1$. This leads to
\begin{equation}\label{cgc-su2-3F2-hahn}
\braket{j_1m_1,j_2m_2}{jm}_q= (-1)^{j_1+j_2-j} \, \widehat{\Gamma}^{j_1,j_2,j}_{m,m_1,m_2}
{_3}F_2 \bigg(\!\begin{array}{c}
	m_2-j_2	\, , \,  m+j+1 \, , \,  m-j   \\
	m-j_1-j_2 \, , \, m+j_1-j_2+1  
\end{array}\!\bigg\vert\, q\, , \, q^{-(j_2+m_2+1)}\!\bigg),
\end{equation}
where, now,
\begin{equation*}
\begin{split}
\widehat{\Gamma}^{j_1,j_2,j}_{m,m_1,m_2}= & q^{(j_2-m_2)(m+j_1+j_2+1)}\frac{(j+j_1-j_2+1\vert q)_{j_2-m_2} 
(-j_1-m_1\vert q)_{j_2-m_2} }{(-2j_2\vert q)_{j_2-m_2}(m_2-j-j_1\vert q)_{j_2-m_2}} \Gamma^{j_1,j_2,j}_{m_1,m_2,m}\\
= & (-1)^{j_2-m_2} q^{(j_2-m_2)(m+j_1+j_2+1)}\frac{\qn{j_1+m_1}!\qn{j_2+m_2}!}{\qn{m+j_1-j_2}\qn{2j_2}} \Gamma^{j_1,j_2,j}_{m_1,m_2,m}.
\end{split}
\end{equation*}

Before writing down the connection with the $q$-Hahn polynomials it is convenient to introduce some properties of such family of $q$-polynomials by means of the $q$-hypergeometric function ${_3}F_2$ instead of the basic series.

Let $N$ a nonnegative integer and the $q$-linear lattice $x(s)=(q^{2s}-1)/(q^2-1)$ with $s=0,1,\dots,N-1$. The $q$-Hahn polynomials on $x(s)$ can be written in terms of the $q$-symmetric hypergeometric function \eqref{q-hip-def} as follows
\begin{small}
\begin{align}
\label{q-hahn1}
h_n^{\alpha,\beta}(x(s),N)_q  =
(-1)^n q^{\frac{n}{2}(\alpha+\beta+\frac{n+1}{2})}(\beta\!+\!1|q)_n\qbc{\!N\!-\!1\!}{n}\!
{_3}F_2\bigg(\!\begin{array}{c}
	-n	\, , \,  -s \, , \,  \alpha+\beta+n+1   \\
	\beta+1 \, , \, 1 -N 
\end{array}\!\bigg\vert\, q\, , \, q^{s-N-\alpha}\!\bigg),
\\  \label{q-hahn2}
= (-1)^n q^{-\frac{n(n-1)}{4}}\frac{(\beta+1|q)_n (N+\alpha+\beta+1|q)_n}{\qn{n}!}
{_3}F_2 \bigg(\!\begin{array}{c}
	-n	\, , \,  s+\beta+1 \, , \,  \alpha+\beta+n+1   \\
	\beta+1 \, , \, N+\alpha+\beta+1  
\end{array}\!\bigg\vert\, q\, , \, q^{s-N+1}\!\bigg), 
\end{align}
\end{small}where it is assumed that $\alpha,\beta>-1$ for orthogonality purposes. To obtain the expression \eqref{q-hahn2} we have corrected a typo in \cite[Eq.~(4.38), p.~32]{arv}. From the above representation it can be shown 
(see \cite[p.~233]{nu-q}) that the Hahn polynomials are polynomials on $x(s)=(q^{2s}-1)/(q^2-1)$ of degree~$n$. Notice also that when $q\to1$, $x(s)\to s$, and the $q$-Hahn polynomials defined in \eqref{q-hahn1}-\eqref{q-hahn2} become into the classical ones~\cite[Eq.~(2.7.19), p.~52]{nsu}.

The $q$-Hahn polynomials are orthogonal with respect to a weight function $\rho$, i.e., 
\begin{equation}
\displaystyle \sum_{s = 0}^{N-1} h_n^{\alpha,\beta}(x(s),N)_q h_m^{\alpha,\beta}(x(s),N)_q \rho(s) 
\bigtriangleup x(s-1/2) = \delta_{n,m}d_n^2,
\end{equation}
where $\rho$ and $d_n^2$ are given in Table~\ref{tabla-hahn}. Hereafter we denote by $\bigtriangleup$ and $\bigtriangledown$ the progressive and regressive finite difference operators
$$
\bigtriangleup f(x)=f(x+1)-f(x) \quad\text{and}\quad \bigtriangledown f(x)=f(x)-f(x-1).
$$
They also satisfy a linear difference equation on $s$ 
\begin{equation} \xi(s)y(s+1)+\left(\lambda_{n}-\zeta(s)-\xi(s)\right)y(s)+\zeta(s)y(s-1)=0,
	\label{eqd-s}
\end{equation}
where
\begin{equation}
\xi(s)=\frac{\sigma(s)+\tau(s)\bigtriangleup x(s-1/2)}{\bigtriangleup x(s-1/2)\bigtriangledown x(s)},\quad\zeta(s)=\frac{\sigma(s)}{\bigtriangleup x(s-1/2)\bigtriangledown x(s)},
\label{dif-eq}
\end{equation}
as well as a three-term recurrence relation (which is also a second order difference equation but in $n$)
\begin{equation}
	\label{ttrr} x(s) h_n^{\alpha,\beta}(x(s),N)_q =\alpha _n h_{n+1}^{\alpha,\beta}(x(s),N)_q +
	\beta_n   P_{n+1}(x(s)) + \gamma_n   h_{n-1}^{\alpha,\beta}(x(s),N)_q , \quad n\geq 0,
\end{equation}
with $ h_{-1}^{\alpha,\beta}(x(s),N)_q=0$ and $ h_0^{\alpha,\beta}(x(s),N)_q=1$, and where again all the involved quantities and functions are given in Table~\ref{tabla-hahn}.

\begin{table}[ht!]
	\caption{\label{tabla-hahn}Main data for the $q$-Hahn polynomials
		in the lattice $x(s)=\frac{q^{2s}-1}{q^2-1}$.}  \vspace{.1cm}
{\small \begin{center}{  \renewcommand{\arraystretch}{.75}
\begin{tabular}{@{}|c|c|}\hline & \\
    &    $h_n^{\alpha,\beta}(x(s),N)_q$\\ 
	& \\
	\hline\hline
	& \\
	$\sigma(s)$ & $-q^{-\frac{N+\alpha}{2}}x(s)^{2}+q^{-\frac{1}{2}}\qn{N+\alpha}x(s)$  \\
	& \\ \hline & \\
	$\sigma(s)+\tau(s)\bigtriangleup x(s-1/2) $ & $-q^{s+\frac{\alpha+\beta}{2}}\qn{s+\beta+1}\qn{s-N+1}$ \\
	& \\ \hline & \\
	$\lambda_n$ & $q^{\frac{\beta+2-N}{2}}\qn{n}\qn{n+\alpha+\beta+1} $\\
	& \\ \hline & \\
	$\rho(s)$ & $\displaystyle q^{(\frac{\alpha+\beta}{2})s}\frac{\tilde{\Gamma}_q(s+\beta+1)
	\tilde{\Gamma}_q(N+\alpha-s)}{\tilde{\Gamma}_q(s+1)\tilde{\Gamma}_q(N-s)}$	\\
	& \\ \hline & \\
	$d_n^{2}$ & $\displaystyle \frac{q^{N(\frac{\alpha+\beta}{2}+N)-\frac{\alpha(\alpha-3)+2}{4}}\tilde{\Gamma}_{q}(n+\alpha+1)\tilde{\Gamma}_{q}(n+\beta+1)\tilde{\Gamma}_{q}(n+\alpha+\beta+N+1)}{q^{n\left(\frac{\alpha-\beta}{2}+N+1\right)}\qn{n}!\qn{N-n-1}!\tilde{\Gamma}_{q}(n+\alpha+\beta+1)\tilde{\Gamma}_{q}(2n+\alpha+\beta+2)}$ \\
	& \\ \hline & \\
	$\alpha _n$ & $\displaystyle q^{-\frac{\beta+2-N}{2}}\frac{\qn{n+1}\qn{n+\alpha+\beta+1}}{\qn{2n+\alpha+\beta+2}\qn{2n+\alpha+\beta+1}}$ \\
	 & \\ \hline & \\
	$\beta_n $ & $\displaystyle q^{\alpha+N+\frac{n}{2}-1}\frac{\qn{n+\alpha+\beta+1}\qn{n+\beta+1}\qn{N-n-2}}{\qn{2n+\alpha+\beta+2}\qn{2n+\alpha+\beta+1}}$ \\
& $\displaystyle+ \frac{q^{-(2N+\beta+n+3(\alpha+1))/2}\qn{n+\alpha}\qn{n+\alpha+\beta+N}\qn{N-n}\qn{n}}{\qn{2n+\alpha+\beta+1}\qn{2n+\alpha+\beta}^{2}\qn{2n+\alpha+\beta-1}\qn{N-n-1}}$ \\
	& \\ \hline & \\
	$\gamma_n$ & $\displaystyle q^{-\frac{N+\alpha}{2}-2}\frac{\qn{n+\alpha}\qn{n+\beta}\qn{n+\alpha+\beta+N}\qn{N-n}}{\qn{2n+\alpha+\beta+1}\qn{2n+\alpha+\beta}^{2}\qn{2n+\alpha+\beta-1}} $ \\
	& \\ \hline
\end{tabular}
}\end{center} }
\end{table}

If we compare the ${_3}F_2$ function in \eqref{cgc-su2-3F2-hahn} with the one in \eqref{q-hahn1} we will see that they coincide if we make the choice
\begin{equation}
s=j_{2}-m_{2},\quad n=j-m, \quad N=j_{1}+j_{2}-m+1,\quad\alpha=m-j_{1}+j_{2},\quad \beta=m+j_{1}-j_{2}.  
\label{parameters-CGC-j2}
\end{equation}
Therefore, it is straightforward to see, by substituting the ${_3}F_2$ function in \eqref{cgc-su2-3F2-hahn} by the one in \eqref{q-hahn1}, that the following relation holds
\begin{equation}
(-1)^{j_1-m_1}\braket{j_1m_1,j_2m_2}{jm}_q=\sqrt{\frac{\rho(s)\bigtriangleup x(s-1/2)}{d_n^2}} 
h_n^{\alpha,\beta}(x(s),N)_q.
\label{CGC-Hahn-J2}
\end{equation}
Notice that this is not the connection formula given in \cite{arv}. To obtain the aforementioned we can do the following: first we use the symmetry relation \eqref{simj1j2} and then use twice the transformation formula \eqref{142q} by means of choosing the appropriate parameters. This leads to a similar expression to  \eqref{cgc-su2-3F2-hahn} but now the hypergeometric function ${_3}F_2$ is given by
$$
{_3}F_2 \bigg(\!\begin{array}{c}
	m_1-j_1	\, , \,  m+j+1 \, , \,  m-j   \\
	m-j_1-j_2 \, , \, m+j_2-j_1+1  
\end{array}\!\bigg\vert\, q\, , \, q^{j_1+m_1+1}\!\bigg).
$$
Comparing the above function with the ${_3}F_2$ given in \eqref{q-hahn1} we obtain the connection formula
\begin{equation}
(-1)^{j_1-m_1+j-m}\braket{j_1m_1,j_2m_2}{jm}_q=\sqrt{\frac{\rho(s)\bigtriangleup x(s-1/2)}{d_n^2}}
h_n^{\alpha,\beta}(x(s),N)_{1/q},
\label{CGC-Hahn-J1}
\end{equation}
where now
\begin{equation}
s=j_{1}-m_{1},\quad N=j_{1}+j_{2}-m+1,\quad\alpha=m+j_{1}-j_{2},\quad \beta=m-j_{1}+j_{2},\quad n=j-m.  \label{parameters-CGC-j1}
\end{equation}
Notice that in \eqref{CGC-Hahn-J1} the $q$-Hahn polynomials are defined for~$q^{-1}$. 

There are several relevant consequences of the above connection formulas. For example,
the three-term recurrence relation \eqref{ttrr} leads to the following one~\cite{smi3}
\begin{equation}
\begin{array}{l}
\sqrt{\dfrac{\qn{j-m}\qn{j+m}\qn{j_1+j_2+j+1}\qn{j_2-j_1+j}\qn{j-j_2+j_1}\qn{j_1+j_2-j+1}}
	{\qn{2j+1}\qn{2j-1}\qn{2j}^{2}}}  \braket{j_1 m_1, j_2 m_2}{j-1\, m}_q  \\
\\ 
\ \ +\sqrt{\dfrac{\qn{j-m+1}\qn{j+m+1}\qn{j_1+j_2+j+2}\qn{j_2-j_1+j+1}\qn{j-j_2+j_1+1}}
	{\qn{2j+3}\qn{2j}\qn{2j+2}^{2}\qn{j_1+j_2-j}^{-1}}} \braket{j_1 m_1, j_2 m_2}{j+1\, m}_q \\
\\
\ \ +\bigg\{\dfrac{\qn{2j}\qn{2j_{1}+2}-\qn{2}\qn{j_{1}+j_{2}-j+1}\qn{j-j_{1}+j_{2}}}{\qn{2}\qn{2j}\qn{2j+2}}
\left(q^{-(j+1)/2}\qn{j+m}-q^{(j+1)/2}\qn{j-m} \right) \\
\\
\ \ -\dfrac{q^{-\frac{m_{2}}{2}}}{\qn{2}}\left(q^{-(j_{1}+1)/2}\qn{j_{1}+m_{1}}-q^{(j_{1}+1)/2}\qn{j_{1}-m_{1}}\right)\bigg\}
\braket{j_1 m_1, j_2 m_2}{j m}_q =0. 
\end{array}
\label{ccgj} 
\end{equation}
and the difference equation \eqref{dif-eq} gives~\cite{smi1}
\begin{equation}  
\begin{array}{l}
q^{-1}\sqrt{\qn{m_2-j_2-1}\qn{j_2+m_2}\qn{m_1-j_1}\qn{j_1+m_1+1}} 
	\braket{j_1\,m_1+1, j_2\, m_2-1}{j m}_q \\
\\
\quad+\sqrt{\qn{m_2-j_2}\qn{j_2+m_2+1}\qn{j_1+m_1}\qn{m_1-j_1-1}}
	\braket{j_1\,m_1-1, j_2\,m_2+1}{j m}_q \\
\\
\quad+\Big(q^{-m_1}\qn{j_2+m_2+1}\qn{j_2-m_2}+q^{m_2}\qn{j_1+m_1+1}\qn{j_1-m_1} \\
\\ 
\quad+\qn{j+1/2 }^2-\qn{m+1/2}^2\Big)q^{-(m_2-m_1+1)/2}
	\braket{j_1 m_1, j_2 m_2}{j m}_q=0,
\end{array}
\label{ttrr-m1-m2}
\end{equation}
which can be used for the numerical computation of the~CGC.

We conclude this section pointing out that the representation \eqref{cgc-su2-3F2-hahn} can be also used for connecting the CGC with the so-called $q$-dual Hahn in the lattice $x(s)=\qn{s}\qn{s+1}$. This have been done in \cite{ran-dual} and will not be considered here. In fact, there are several other relations involving the $q$-Hahn polynomials different from the \eqref{dif-eq} and \eqref{ttrr} that lead to recurrence relations for the CGC. The interested reader is referred to the aforementioned papers~\cite{ran-dual,arv}.

\section*{Concluding remarks}
As we see, the use of the special function theory, in this case the $q$-hypergeometric functions (or, equivalently, the basic series) can be very useful for the study of the Clebsch-Gordan coefficients of the $su_q(2)$ algebra. As it has been shown, there are several properties that are quite complicated to obtain by using representation theory tools, but they can be easily obtained exploiting the representation of the CGC in terms of the symmetric $q$-hypergeometric function ${_3}F_2$ (see formula \eqref{cgc-su2-3F2-long}). In particular, classical results for the algebra (group) $su(2)$ are obtained from the given ones in this work just by taking the limit $q\to1$. It is worth pointing out that a similar analysis can be done also for the  non compact algebra $su_q(1,1)$. The results for this case will be published elsewhere.


\section*{Statements and declarations}
\paragraph{\textbf{Funding}}
Renato \'Alvarez-Nodarse was partially supported by PGC2018-096504-B-C31 (FEDER (EU) / Mi\-nis\-te\-rio de Cien\-cia e In\-no\-va\-ci\'on - A\-gen\-cia Es\-ta\-tal de In\-ves\-ti\-ga\-ci\-\'on), FQM-262, and Feder-US-1254600 (FEDER (EU) - Jun\-ta de An\-da\-lu\-c\'i\-a). Alberto Arenas-G\'omez was partially supported by PGC2018-096504-B-C32 (FEDER (EU) / Mi\-nis\-te\-rio de Cien\-cia e In\-no\-va\-ci\'on - A\-gen\-cia Es\-ta\-tal de In\-ves\-ti\-ga\-ci\-\'on).

%
%

\end{document}